\documentclass[a4paper]{article}
\date{\today}
\usepackage[dvipdfmx]{graphicx}
\usepackage{latexsym}
\usepackage{amssymb}
\usepackage{amsmath}
\usepackage{amscd}
\usepackage{amsthm}
\usepackage{amsfonts}
\usepackage{mathrsfs}
\usepackage{enumerate}
\usepackage{enumitem}
\usepackage{bm}
\usepackage[usenames]{color}
\usepackage[active]{srcltx}
\usepackage{amsthm}
\theoremstyle{plain}
 \newtheorem{theorem}{Theorem}[section]

 \newtheorem{proposition}[theorem]{Proposition}
 \newtheorem{fact}[theorem]{Fact}
 \newtheorem*{fact*}{Fact}
 \newtheorem{lemma}[theorem]{Lemma}
 \newtheorem{corollary}[theorem]{Corollary}
 \theoremstyle{remark}
 \newtheorem{definition}[theorem]{Definition}
 \newtheorem{remark}[theorem]{Remark}
 
 \newtheorem{example}[theorem]{Example}
\numberwithin{equation}{section}

 \makeatletter
    
    \@addtoreset{equation}{section}
  \makeatother
\newcommand{\CC}{\Psi}
\newcommand{\I}{g}
\renewcommand{\t}{t}

\newcommand{\RR}{U}
\newcommand{\dR}{V}

\newcommand{\ivar}{\zeta}

\newcommand{\R}{\boldsymbol{R}}

\newcommand{\ga}{\gamma}

\newcommand{\vect}[1]{\boldsymbol{#1}}

\newcommand{\rank}{\operatorname{rank}}

\begin{document}
\title{Bour's theorem for helicoidal surfaces with singularities}

\author{
Yuki Hattori\\
\vspace{2mm}
{\it\small Graduate School of Engineering Science, Yokohama National University,}
\vspace{-2mm}\\
{\small\it Hodogaya, Yokohama 240-8501, Japan}\\
{\small\tt hattori-yuki-by@ynu.jp}\\
{\small https://orcid.org/0009-0003-4934-4365}
\vspace{6mm}\\
Atsufumi Honda\footnote{Corresponding author.}\\
\vspace{2mm}
{\it\small Graduate School of Engineering Science, Yokohama National University,}
\vspace{-2mm}\\
{\it\small Hodogaya, Yokohama 240-8501, Japan}\\
{\small\tt honda-atsufumi-kp@ynu.ac.jp}\\
{\small https://orcid.org/0000-0003-0515-2866}
\vspace{6mm}\\
Tatsuya Morimoto\\
\vspace{2mm}
{\small\it College of Engineering Science, Yokohama National University,}
\vspace{-2mm}\\
{\it\small Hodogaya, Yokohama 240-8501, Japan}\\
{\small {\it Current address}: Froebel-Kan Co.,Ltd, 
   Honkomagome, Tokyo 113-8611, Japan}\\
{\small\tt morimoto.tatsuya.7@gmail.com}
}


\maketitle

\begin{abstract}
In this paper, 
generalizing the techniques of 
Bour's theorem, 
we prove that
every generic cuspidal edge,
more generally, generic $n$-type edge,
which is invariant
under a helicoidal motion in Euclidean $3$-space
admits non-trivial isometric deformations.
As a corollary,
several geometric invariants, 
such as 
the limiting normal curvature,
the cusp-directional torsion,
the higher order cuspidal curvature
and the bias, 
are proved to be extrinsic invariants.
\end{abstract}

\renewcommand{\thefootnote}{\fnsymbol{footnote}}

\section{Introduction}
A surface in the Euclidean $3$-space $\R^3$ is called 
a {\it helicoidal surface} 
if it is invariant under a helicoidal motion.
Here, a helicoidal motion is defined as 
a non-trivial one-parameter subgroup 
of the isometry group of $\R^3$.
Helicoidal surfaces are a generalization of surfaces of revolution.
In 1862, E.\ Bour proved that,
for a given helicoidal surface,
there exists a two-parameter family of 
helicoidal surfaces which are isometric to it
\cite[Theorem II, p.\ 82]{Bour}.
The procedure is as follows.
Let $f:\Sigma \to \R^3$ be a helicoidal surface.
First, 
Bour proved the existence of a special orthogonal coordinate
system $(s,t)$, called {\it natural parameters},
on a source $2$-manifold $\Sigma$
such that the families of $s$-coordinate curves are geodesic on $\Sigma$
parametrized by arc length,
and the families of the images of $t$-coordinate curves 
are the trajectories of the helicoidal motion.
In this natural coordinate system $(s,t)$, 
the first fundamental form $g$ is written as
\begin{equation}\label{eq:metric-Bour}\tag{B$_0$}
  g=ds^2+U(s)^2 \,dt^2,
\end{equation} 
where $U(s)$ is a positive smooth function.
Conversely, 
for given natural parameters $(s,t)$ of $\Sigma$
and a positive smooth function $U(s)$,
Bour determined a $2$-parameter family of 
helicoidal surfaces whose first fundamental form 
coincides with \eqref{eq:metric-Bour}.
See also \cite[pp. 129--130]{Darboux}.
Do Carmo and Dajczer used Bour's theorem 
to obtain the integral representation of helicoidal surfaces of 
constant mean curvature \cite{doCarmo-Dajczer}.

Bour's theorem has been shown to hold 
even if the ambient space is a more general 
Riemannian or Lorentzian $3$-manifold;
The Lorentz-Minkowski space $\R^3_1$
by Sasahara \cite{Sasahara}, 
Ji-Kim \cite{Ji-Kim}
and the second author \cite{Honda};
The product spaces $S^2\times \R$, $H^2\times \R$ 
and the Heisenberg group ${\rm Nil}_3$ by Sa Earp \cite{SaEarp}.
Sa Earp--Toubiana \cite{SaEarp-Toubiana}
generalized the results in \cite{SaEarp};
The Bianchi-Cartan-Vranceanu (BCV) space
by Caddeo--Onnis--Piu \cite{COP}.
Here, the BCV space
is a Riemannian $3$-manifold with $4$-dimensional isometry group,
together with simply connected space forms of 
non-negative sectional curvature.
Recently, Domingos--Onnis--Piu \cite{DOP}
generalized Bour's theorem 
for surfaces that are invariant under the action of 
a one-parameter group of isometries of 
a Riemannian $3$-manifold.

In this paper, we give an intrinsic generalization of Bour's theorem.
Our first fundamental form $g$ is written as
\begin{equation}\label{eq:metric-Bour-k}\tag{B$_k$}
  g=s^{2k}\,ds^2+U(s)^2 \,dt^2,
\end{equation} 
where $k$ is a positive integer
and $d^{i}\RR/ds^{i}(0)=0$ holds for 
all $i$ with $1\le i \le k$.
Hence, we deal with helicoidal surfaces with 
{\it singularities}.
More precisely, 
let $f : \Sigma\to \R^3$
be a smooth map 
defined on a smooth $2$-manifold $\Sigma$.
A point $p\in \Sigma$ is called a \emph{singular point\/} 
where $f$ is not an immersion.
We set a map $f_C : \R^2\to \R^3$ by
$f_C(u,v):=(u^2,u^3,v)$,
which is called the {\it standard cuspidal edge\/}
(see Figure \ref{fig:CE}, left).
A singular point $p\in \Sigma$ of a smooth map
$f : \Sigma\to \R^3$ is said to be 
{\it cuspidal edge}, if 
the map-germ $f : (\Sigma,p)\to (\R^3,f(p))$
is $\mathcal{A}$-equivalent to 
the map-germ $f_C : (\R^2,0)\to (\R^3,0)$.
Here, the $\mathcal{A}$-equivalence relation is defined as follows.
Let $N$ be a positive integer.
For $i=1,2$, 
let $f_i : \Sigma_i\to \R^N$ be smooth maps,
where $\Sigma_i$ $(i=1,2)$
are smooth manifolds.
For points $p_i\in \Sigma_i$ $(i=1,2)$,
we set $q_i=f_i(p_i)$.
Two map-germs 
$f_i : (\Sigma_i,p_i)\to (\R^N,q_i)$
are said to be {\it $\mathcal{A}$-equivalent\/}
if there exist diffeomorphism-germs
$\varphi : (\Sigma_1,p_1) \to (\Sigma_2,p_2)$ and
$\Phi : (\R^N,q_1) \to (\R^N,q_2)$
such that $f_2=\Phi\circ f_1\circ \varphi^{-1}$.

\begin{figure}[htb]
\begin{center}
 \begin{tabular}{{c@{\hspace{1cm}}c}}
  \resizebox{4.5cm}{!}{\includegraphics{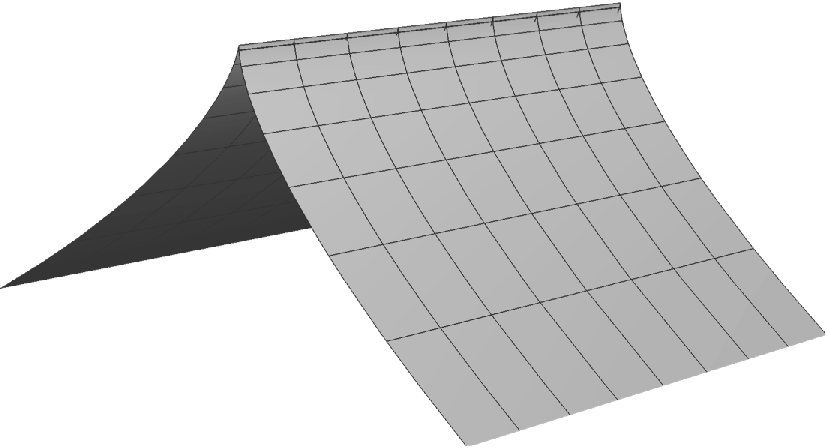}}&
  \resizebox{3.5cm}{!}{\includegraphics{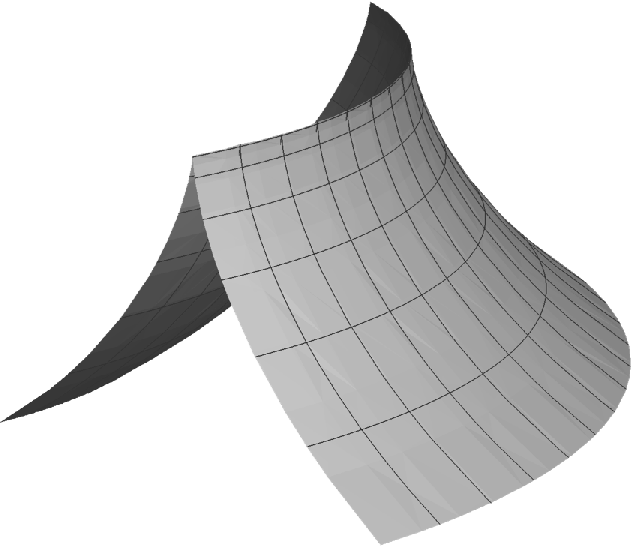}} 
 \end{tabular}
\end{center}
\caption{Left: the image of
the standard cuspidal edge $f_C$.
Right: the image of
a generic cuspidal edge
(cf.\ Example \ref{ex:GHF}).
}
\label{fig:CE}
\end{figure}

In \cite{SUY-ann},
the singular curvature $\kappa_s$ 
and 
the limiting normal curvature $\kappa_\nu$
for cuspidal edge singular points
are introduced (cf.\ Definition \ref{def:generic}).
These invariants play important roles
when we investigate cuspidal edges
from the differential geometric view point.
For example, 
the sign of $\kappa_s$ affects
the shape \cite[Theorem 1.17]{SUY-ann}.
Moreover, 
$\kappa_s$ appears in the remainder term
of the Gauss-Bonnet type formula
\cite{Kos2, SUY-ann, SUY-kyushu, SUY4}.
The singular curvature $\kappa_s$
is known to be an {\it intrinsic invariant}
\cite[Proposition 1.8]{SUY-ann},
namely, $\kappa_s$ can be expressed in terms of 
the first fundamental form.
The limiting normal curvature
$\kappa_{\nu}$
is closely related to the behavior of 
the Gaussian curvature $K$
(\cite[Theorem 3.1]{SUY-ann}, \cite[Theorem 3.9]{MSUY}).

Unlike the singular curvature,
it is proved in \cite[Corollary B]{NUY}
that the limiting normal curvature
$\kappa_{\nu}$ is an {\it extrinsic invariant}.
More precisely,
Naokawa--Umehara--Yamada \cite{NUY}
extended the classical Janet--Cartan's 
local isometric embedding theorem 
to cuspidal edges, 
which yields isometric deformations of 
real analytic `generic' cuspidal edges
that change $\kappa_{\nu}$.
Here, a {\it generic cuspidal edge} is 
a cuspidal edge with non-vanishing 
limiting normal curvature, that is, $\kappa_{\nu}\ne0$
\cite{SUY-ann},
see also Definition \ref{def:generic}.
See also \cite{Kos, HNUY}
(cf.\ \cite{HHNSUY}).
The number of congruence classes
of germs of real analytic generic cuspidal edges
having the common first fundamental form
and the common singular value set
is determined in \cite{HNSUY1} 
(cf.\ Remark \ref{rem:isomer}).
See also \cite{HNSUY2}.
These techniques used in cuspidal edges
are also applied to curved foldings
\cite{HNSUY3, HNSUY4}
(cf.\ \cite{Fuchs-Tabach}).

Let $n$ be a positive integer.
As a generalization of cuspidal edge,
Martins--Saji--Santos--Teramoto \cite{MSST}
introduced a class of singular points called 
{\it $n$-type edge\/}
whose germ is $\mathcal{A}$-equivalent to 
$
(u,v)\mapsto (u^n,u^{n+1}a(u,v),v)
$
for a smooth function $a(u,v)$.
When $n=2$, 
we remark that
$2$-type edge singular points
are just generalized cuspidal edge singular points
\cite{HNSUY1}.

In this paper, we prove 
a Bour-type theorem
for generic helicoidal $n$-type edges
(Theorems \ref{thm:CE-3} and \ref{thm:singular-Bour}).
Here, a {\it generic helicoidal $n$-type edge\/}
is an $n$-type edge
which is invariant under a helicoidal motion
of $\R^3$ and $\kappa_\nu\ne0$
(Definition \ref{def:generic-helicoidal}).
As a corollary,
we obtain a Bour-type theorem
for generic helicoidal $r/n$-cuspidal edges
(Corollaries \ref{cor:Bour-CE} and \ref{cor:Bour-45CE}).
In our theorem, 
the first fundamental form coincides with 
\eqref{eq:metric-Bour-k}
for $n=k+1$.
We remark that 
$s$ is not an arc length parameter,
but the `canonical parameter' at cusps
(Shiba--Umehara \cite{su}, see also \cite{MS2, fukui2},
cf.\ Fact \ref{fact:PBAL-riem}).
We also remark that our Bour-type theorem
do not need the real analyticity.
In \cite{MSST},
several geometric invariants 
of $n$-type edge are introduced,
such as
the {\it cusp-directional torsion\/} $\kappa_t$ 
(cf.\ \cite{MS}),
the {\it $(n,n+i)$-cuspidal curvature\/} $\omega_{n,n+i}$ 
(cf.\ \cite{MSUY}),
and the {\it $(n,2n)$-bias\/} $\beta_{n,2n}$
(cf.\ \cite{Honda-Saji}).
As a corollary of our Bour-type theorem
(Theorems \ref{thm:CE-3} and \ref{thm:singular-Bour}),
we prove the extrinsicity of these invariants 
(Corollaries \ref{cor:nu-c} and \ref{cor:kappa-c}).

This paper is organized as follows.
In Section \ref{sec:wave-front},
we review some basic definitions
and known results
for curves or surfaces with singularities.
In Section \ref{sec:sing-Bour},
we prove a Bour-type theorem
for generic helicoidal $n$-type edges in $\R^3$
(Theorems \ref{thm:CE-3} and \ref{thm:singular-Bour}).
In Section \ref{sec:application},
we give applications of our Bour-type theorem.
In subsection \ref{sec:nr-ce},
we prove a Bour-type isometric deformation theorem
for generic helicoidal $r/n$-cuspidal edges
(Corollaries \ref{cor:Bour-CE} and \ref{cor:Bour-45CE}).
On the other hand, 
by using our Bour-type isometric deformation theorem, 
the existence of isometric deformations 
that do not preserve the singularity types is shown,
see Corollary \ref{cor:rigidity}.
In subsection \ref{sec:extrinsic},
we prove that several invariants, 
such as the limiting normal curvature $\kappa_\nu$,
the cusp-directional torsion $\kappa_t$,
the $(n,n+i)$-cuspidal curvature $\omega_{n,n+i}$,
and the $(n,2n)$-bias $\beta_{n,2n}$
of generic helicoidal $n$-type edges are extrinsic 
(Corollaries \ref{cor:nu-c} and \ref{cor:kappa-c}).
In Appendix \ref{sec:7/2}, 
we introduce a criterion for $7/2$-cusp
based on the work by Takagi \cite{Takagi}.

Throughout the paper, 
manifolds and maps are differentiable of class $C^\infty$ 
unless otherwise indicated.

\section{Preliminaries}
\label{sec:wave-front}

In this section, we review the fundamental properties 
of curves and surfaces with singularities 
in Euclidean $3$-space $\R^3$.

\subsection{Wave fronts}
\label{sec:prelim-CE}

Let $\Sigma$ be an oriented smooth $2$-manifold.
For a smooth map $f:\Sigma \to \R^3$,
we denote its singular set by $S(f)$.
In this paper, we say $f$ a {\it surface}
if the regular set $\Sigma \setminus S(f)$ is dence.
If $S(f)$ is an empty set,
$f$ is said to be a {\it regular surface}.
We denote by $S^2$ the unit sphere 
$$S^2:=\left\{ \vect{x}\in \R^3\,;\, \|\vect{x}\|=1 \right\}$$
in $\R^3$,
where $\|\vect{x}\|=\sqrt{\vect{x}\cdot \vect{x}}$ and 
the dot implies the canonical inner product of $\R^3$.
A surface $f:\Sigma \to \R^3$ is called a \emph{frontal} 
if, for each point $p\in \Sigma$,
there exist an open neighborhood $D$ of $p$
and a smooth map 
$\nu : D\rightarrow S^2$ 
such that 
$df_q(\vect{v})\cdot \nu(q)=0$
holds for all $\vect{v}\in T_q\Sigma$, $q\in D$.
The map $\nu$ is called
a \emph{unit normal vector field} along $f$.
If $L:=(f,\nu):D\to \R^3\times S^2$ is an immersion, 
$f$ is called a \emph{wave front} 
(or a \emph{front}, for short).
For a more detailed explanation
of the geometric properties of wave fronts or frontals, 
please refer to the book \cite{SUY-book}.

\subsection{Geometric invariants}
\label{sec:invs}
Let $f:\Sigma \to \R^3$ be a surface,
and take a point $p\in \Sigma$.
A real number $J(f,p)\in \R$
determined by $f$ and $p$
is said to be a {\it geometric invariant} of $f$ at $p$
if 
$
  \left|J(f,p) \right|
  =
  \left|J(T\circ f\circ \varphi^{-1},p)\right|
$
holds for any isometry $T$ of $\R^3$
and any diffeomorphism $\varphi:D\to D$
defined on a sufficiently small 
open neighborhood $D$ of $p$
such that $\varphi(p)=p$.

We denote by 
$\I$, or $\I_f$,
the {\it first fundamental form} of $f$.
Let $\bar{f} :\Sigma \to \R^3$
be a surface.
Then, we say 
{\it $\bar{f}$ is isometric to $f$ at $p$}
if there exist 
an open neighborhood $D$ of $p$
and a diffeomorphism $\varphi:D\to D$
such that 
$\varphi(p)=p$ and 
$\I_f = \varphi^*\I_{{\bar{f}}}$
holds on $D$.
A geometric invariant $J(f,p)$
is said to be an {\it intrinsic invariant}
if $$|J(f,p)|=|J(\bar{f},p)|$$ holds for 
any surface $\bar{f}$ which is isometric to $f$ at $p$.
A geometric invariant $J(f,p)$
is said to be an {\it extrinsic invariant}
if it is not an intrinsic invariant.
Namely, there exist an open neighborhood 
$D$ of $p$,
and a surface
$\bar{f}: D \to \R^3$
such that
$\bar{f}$ is isometric to $f$ at $p$,
but 
$|J(f,p)|\ne |J(\bar{f},p)|$ holds.


Let $f:\Sigma \to \R^3$ be a surface,
and fix a point $p\in \Sigma$.
A continuous $1$-parameter family of maps
$\{f^t:D\to \R^3 \}_{t\in [0,1]}$
defined on an open neighborhood $D$ of $p$
is called an {\it isometric deformation of $f$ at $p$},
if 
\begin{itemize}
\item $f^0=f$, and 
\item $f^t$ is isometric to $f$ at $p$ for each $t\in [0,1]$.
\end{itemize}
Let $\{\varphi^t\}_{t\in [0,1]}$
be a continuous $1$-parameter family of 
diffeomorphisms defined on $D$,
and $\{T^t\}_{t\in [0,1]}$ 
be a continuous $1$-parameter family of isometries of $\R^3$.
Suppose that both $\varphi^0$ and $T^0$ are the identity maps.
Then, 
$\{f^t\}_{t\in [0,1]}$
defined by
$f^t:=T^t\circ f \circ (\varphi^t)^{-1}$
is said to be a 
{\it trivial isometric deformation}.
By definition,
geometric invariants of surfaces
are preserved 
under trivial isometric deformations.

A singular point $p$ of $f$
is called a {\it rank one singular point}
if ${\rm rank}(df_p)=1$.

\begin{definition}[\cite{SUY-ann, MSUY}]
\label{def:generic}
Let $f:\Sigma\to\R^3$ be a frontal
and $p$ be a rank one singular point of $f$.
\begin{enumerate}
\item
A coordinate neighborhood $(D;u,v)$ of $p$
which is compatible with 
respect to the orientation of $\Sigma$
is said to be {\it adjusted} at $p$
if $f_u(p)=\vect{0}$.
\item
A real number $\kappa_\nu(p)$ defined by 
\begin{equation}\label{eq:LNC}
  \kappa_\nu(p)
  :=\frac{f_{vv}(p)\cdot \nu(p)}{\|f_v(p)\|^2}
\end{equation}
is called the {\it limiting normal curvature} at $p$.
\item
A rank one singular point with non-zero 
limiting normal curvature $\kappa_{\nu}\ne0$
is called a {\it generic rank one singular point}
(cf.\ \cite{SUY-ann}).
\end{enumerate}
\end{definition}

It is proved in \cite[Proposition 2.9]{MSUY}
that the definition of the limiting normal curvature 
$\kappa_\nu(p)$
does not depend on the choice of 
the coordinate system $(u,v)$ adjusted at $p$.
Hence, $\kappa_\nu(p)$ is a geometric invariant.
It is proved in \cite{SUY-ann} that
the Gaussian curvature $K$ diverges
at generic cuspidal edge singular points.

\subsection{Singular points of curves with finite multiplicities}
Let $I$ be an open interval and $N$ a positive integer.
If a smooth map
$\gamma:I\to \R^N$
admits only isolated singular points,
then we call $\gamma$ a {\it curve}.
Take a point $u_0\in I$.
For a positive integer $n$, 
a smooth curve $\gamma:I\to \R^N$ 
is said to be of {\it multiplicity $n$
at} $u_0$ if 
$$
  \gamma'(u_0)=\cdots = \gamma^{(n-1)}(u_0)=\vect{0},\qquad
  \gamma^{(n)}(u_0)\ne\vect{0}
$$
hold,
where $\gamma^{(i)}=d^i \gamma/d u^i$ for a positive integer $i$.
Such a point $u_0\in I$ is called a 
{\it singular point of multiplicity $n$}.
We say that $\gamma$ is of {\it finite multiplicity at} $u_0$ if 
$\gamma$ is of multiplicity $n$ at $u_0$ 
for some positive integer $n$.
By the Hadamard lemma (cf.\ \cite[Lemma 3.4]{BG}),
there exists a smooth map $\bar{\gamma}:I\to \R^N$
such that
$$
  \gamma(u) - \gamma(u_0) = (u-u_0)^n \bar{\gamma}(u)
  \qquad
  (\bar{\gamma}(u_0)\ne \vect{0}).
$$

The following gives 
the `canonical parameter'
for curves of finite multiplicity.

\begin{fact}[{\cite{su}, \cite{MS2}, \cite{fukui2}}]
\label{fact:PBAL-riem}
For a non-negative integer $k$,
we set $n=k+1$.
Let $\gamma(u)$ $(u\in I)$ be a smooth curve in $\R^N$,
and let $u_0\in I$ be a singular point of multiplicity $n$.
Then there exists a coordinate change
$s=s(u)$ of orientation preserving diffeomorphism such that 
$s(u_0)=0$ and
$$\left\|\frac{d\gamma}{ds}\right\|=|s|^{k}.$$
\end{fact}

We remark that if $k=0$, 
the parameter $s$ is just the arc length parameter.
Fact \ref{fact:PBAL-riem} is proved in
\cite[Theorem 1.1 and Remark 2.2]{su}
in the case of order $2$ (i.e.\ $k=1$) in $\R^2$.
Such the assertion for 
curves of order $2$ in $\R^3$ 
is proved in \cite[Appendix A]{MS2}.
For curves with finite multiplicity in $\R^N$,
see \cite[Theorem 1.1]{fukui2}.

\subsection{Finite type edge singular points}
Let $n$ be a positive integer.
A smooth map $f : \Sigma \to \R^3$ 
has {\it $n$-type edge\/} at a point $p\in \Sigma$, if 
there exists a smooth function $a(u,v)$
such that
the map-germ $f:(\Sigma,p)\to(\R^3,f(p))$ is $\mathcal{A}$-equivalent
to the map-germ
$$
\R^2\ni (u,v)\longmapsto (u^n,u^{n+1}a(u,v),v)\in \R^3
$$ at the origin \cite{MSST}.
Then, $p$ is called an {\it $n$-type edge singular point\/}.
We say that $f$ has {\it finite type edge} at $p$ if 
$p$ is an $n$-type edge singular point
for some positive integer $n$.
We remark that
generalized cuspidal edge singular points,
introduced in \cite{HNSUY1},
are just $2$-type edge singular points.

If $p$ is a rank one singular point,
that is, $\rank (df_p)=1$,
there exists a non-vanishing vector field 
$\eta$ defined on a neighborhood $\tilde{D}$ of $p$
such that $df_q(\eta_q)=\vect{0}$ holds
for each $q\in S(f)\cap \tilde{D}$.
We call $\eta|_{S(f)}$ a {\it null vector field\/}, 
and $\eta$ an {\it extended null vector field\/}.
An extended null vector field is also called a null vector field
if it does not induce a confusion.

We assume that the set of singular points $S(f)$ is a regular curve,
and the tangent direction of $S(f)$ is not in $\ker(df_p)$.
Let $\xi$ be a vector field such that $\xi_p$ is a non-zero
tangent vector of $S(f)$ for $p\in S(f)$.
We consider the following conditions for $(\xi,\eta)$:
\begin{align}
\label{itm:cri1}
&\eta^{i}f=0\quad \text{holds on} ~ S(f) 
\qquad
(i=2,\dots,n-1),\\
\label{itm:cri2}
&\xi f(p)\quad \text{and} \quad
\eta^{n}f(p) \quad\text{are linearly independent.}
\end{align}
Here, 
for a vector field $\zeta$ and a map $f$,
the symbol
$\zeta^i f$ stands for the $i$-times directional derivative
of $f$ by $\zeta$.
Moreover, for a coordinate system $(u,v)$ and a map $f$,
the symbol $f_{v^i}$
stands for $\partial^if/\partial v^i$.

\begin{fact}
[{\cite[Proposition 2.2]{MSST}}]
\label{fact:MSST}
Let $f : \Sigma \to \R^3$ be a frontal
having a rank one singular point $p\in \Sigma$.
Suppose that the singular set $S(f)$ is a regular curve,
and the tangent direction of $S(f)$ generated by $\xi$ is not in 
$\ker(df_p)$.
Then, $f$ has $n$-type edge at $p$
if and only if there exists a null vector field $\eta$ satisfying
\eqref{itm:cri1} and \eqref{itm:cri2}.
\end{fact}

See also \cite{KRSUY} for the criterion of cuspidal edge singular points.
The following lemma was used to prove Fact \ref{fact:MSST}.

\begin{fact}
[{\cite[Lemma 2.4]{MSST}}]
\label{fact:MSST-h}
Suppose that $f$ and $(\xi,\eta)$ satisfy 
the condition \eqref{itm:cri1}
$($respectively, \eqref{itm:cri1} and \eqref{itm:cri2}$)$.
For a non-zero function  $h$,
we set $\hat\eta=h\eta$.
Then
$f$ and $(\xi,\hat\eta)$ satisfy 
the condition \eqref{itm:cri1}
$($respectively, \eqref{itm:cri1} and \eqref{itm:cri2}$)$.
\end{fact}

In \cite{MSST}, several invariants 
of frontals at $n$-type edge singular points
are introduced.
We here review the definition
of the {\it cusp-directional torsion} $\kappa_t$\footnote{%
The cusp-directional torsion was 
originally defined for cuspidal edge singular points
\cite{MS}.}.
More precisely, for a frontal $f : \Sigma \to \R^3$
having $n$-type edge at $p\in \Sigma$,
let $\xi$, $\eta$ be vector field satisfying 
\eqref{itm:cri1} and \eqref{itm:cri2}.
Then, 
the cusp-directional torsion $\kappa_t$
is defined by
$$
   \kappa_t (p) 
  := 
  \dfrac{\det(\xi f,\eta^n f,\xi\eta^n f)}
  {\|\xi f\times\eta^n f\|^2}(p)
  -
  \dfrac{(\xi f\cdot\eta^n f)\det(\xi f,\eta^n f,\xi^2 f)}
  {\|\xi f\|^2\|\xi f\times\eta^n f\|^2}(p).
$$

\section{Generalized Bour's theorem}
\label{sec:sing-Bour}

A helicoidal motion is defined as 
a non-trivial one-parameter subgroup 
of the isometry group of $\R^3$.
A surface in $\R^3$ is called 
a {\it helicoidal surface} 
if it is invariant under 
a helicoidal motion.
A helicoidal motion fixes a line,
which is called the {\it axis}.

Let $x(u)$, $z(u)$ be smooth functions 
defined on an open interval $I$.
We set $\gamma(u):=(x(u),z(u))$.
For a real number $h$,
set $f_{\gamma,h}:I\times \R \to \R^3$ by
\begin{equation}\label{cuspsur}
  f_{\gamma,h}(u,v)
  :=( x(u) \cos v, x(u) \sin v, z(u)+hv)
  \quad(u\in I,~v\in \R).
\end{equation}
Then, $f_{\gamma,h}$ is a helicoidal surface.
The constant $h$ is called the {\it pitch}.
If $h=0$, then $f_{\gamma,h}$ is a surface of revolution.

\begin{definition}
\label{def:generic-helicoidal}
Let $n$ be a integer greater than $1$.
\begin{enumerate}
\item
A frontal in $\R^3$
is said to be a 
{\it generic helicoidal $n$-type edge}
if it is invariant under 
a helicoidal motion,
the singular set is non-empty, and
every singular point is 
a generic $n$-type edge singular point
(cf.\ Definition \ref{def:generic}).
\item
A generic helicoidal $n$-type edge
whose singular set consists of 
$r/n$-cuspidal edges
is called 
a {\it generic helicoidal $r/n$-cuspidal edge}.
\end{enumerate}
\end{definition}

Our main theorem in this paper
is a generalization of Bour's theorem
to generic helicoidal $n$-type edges;
the Bour-type representation formula
(Theorem \ref{thm:CE-3}),
and 
the Bour-type isometric deformation theorem
(Theorem \ref{thm:singular-Bour})
for generic helicoidal $n$-type edges.
For the proof of Theorem \ref{thm:CE-3},
we prepare Lemmas \ref{lem2} and \ref{lem:natural-coord}.
Similarly, 
for the proof of Theorem \ref{thm:singular-Bour},
we prepare Lemmas \ref{lem:LNC} and \ref{lem:cuspidal-torsion}.

\begin{lemma}\label{lem2}
Let $f_{\gamma,h}:I\times \R\to \R^3$
be a helicoidal frontal 
given by \eqref{cuspsur},
and $k$ be a positive integer.
Set $n=k+1$.
If
$p=(u_0,v_0)\in I\times \R$ is 
a generic $n$-type edge singular point
of $f_{\gamma,h}$, then
\begin{equation}\label{eq:heli-A4}
\begin{split}
  &x^{(i)}(u_0)=z^{(i)}(u_0)=0 \quad (i = 1,2, \dots ,k)\\ 
  &x(u_0) \ne 0,\qquad
  z^{(k+1)}(u_0)\ne0
\end{split}
\end{equation}
hold,
where the $x^{(i)}$ means $d^{i}x/du^{i}$.
\end{lemma}

\begin{proof}
The singular set $S(f_{\gamma,h})$ is given by
\begin{equation}\label{eq:Sf}
  S(f_{\gamma,h}) 
  = \{(u,v) \in I \times \R \mid h^2\dot{x}(u)^2=x(u)^2(\dot{x}(u)^2+\dot{z}(u)^2)=0 \},
\end{equation}
where the dot means $d/du$.
Since $(u_0,v_0)$ is a singular point,
$(u_0,v)$ is also a singular point 
for each $v\in \R$.
Hence, $\sigma(v):=(u_0,v)$ $(v\in \R)$ 
parametrizes the connected component of $S(f_{\gamma,h})$
including $p$.

We abbreviate $f_{\gamma,h}$ to $f$,
and set 
$f_u:=\partial f/\partial u$,
$f_v:=\partial f/\partial v$.
First, we prove $x(u_0) \ne 0$ by contradiction.
Suppose that $x(u_0) = 0$ holds.
In the case of $h = 0$, 
since $f_{v}(u_0,v_0) = \vect{0}$ holds,
$\eta = \partial/\partial v$ gives a null vector field.
As $\sigma(v) = (u_0,v)$ is a singular curve, 
singular direction coincides with the null direction.
But, it contradicts that $p$ is an $n$-type edge. 
Hence, $h \ne 0$.
Then, 
$\hat{\sigma}(v)  = f \circ \sigma(v) $
is given by
$
  \hat{\sigma}(v) 
  =f(u_0,v) 
  =\left( 0 , 0 , z(u_0) + hv \right).
$
Hence we have 
$\hat{\sigma}''(v)=\vect{0}$, 
where the prime means $d/dv$.
This implies 
$\kappa_{\nu}(p) 
= {\hat{\sigma}''(v_0)\cdot \nu(p)}/{\|\hat{\sigma}'(v_0)\|^2}
= 0,$
which is a contradiction.
Therefore, $x(u_0) \ne 0$.

Second, we prove $x^{(i)}(u_0)=z^{(i)}(u_0)=0$ 
for all $i$ with $1\le i \le k$.
Without loss of generality, we may suppose that $u_0 = 0$. 
Since $x(0) \ne 0$ holds, \eqref{eq:Sf} yields that 
$\dot{x}(0)=\dot{z}(0)=0$,
and hence, $f_u(0,v)=\vect{0}$.
By Fact \ref{fact:MSST},
there exists a null vector field $\eta$ 
satisfying \eqref{itm:cri1} and \eqref{itm:cri2}.
Since $\partial/\partial u$ points the null direction at $(0,v)$,
we have that $\eta$ is written as 
$$a_1(u,v)\partial_u+a_2(u,v)\partial_v\qquad
(a_1(0,v)\ne0, a_2(0,v)=0),$$
where we set 
$\partial_u=\partial/\partial u$ and $\partial_v=\partial/\partial v$.
By Fact \ref{fact:MSST-h},
dividing by $a_1$,
we may assume that 
$$
  \eta = \partial_u+a(u,v) \partial_v
  \qquad
  (a(0,v)=0).
$$
Then we have $\eta^2f=f_{uu}+a_uf_v$ along the $v$-axis.
Since $\eta^2f(0,v)=\vect{0}$ holds (when $k>1$),
we have 
\begin{align*}
  &\ddot{x}(0)\cos v -a_u(0,v)x(0)\sin v=0,\\
  &\ddot{x}(0)\sin v +a_u(0,v)x(0)\cos v=0,\\
  &\ddot{z}(0)+a_u(0,v)h=0.
\end{align*}
Since $x(0)\ne0$, it holds that $\ddot{x}(0)=\ddot{z}(0)=a_u(0,v)=0$.
Continuing this argument, we have
$x^{(i)}(0)=z^{(i)}(0)=0$ 
and 
$a_{u^{i-1}}(0,v)=0$
for all $i$ with $1\le i \le k$,
where $a_{u^{i-1}}= \partial^{i-1} a/\partial u^{i-1}$.
Then, 
$\eta^{k+1}f=f_{u^{k+1}}+a_{u^{k}}f_v$
holds along the $v$-axis.
Since $f_v$ and $\eta^{k+1}f$ are linearly independent,
$f_{u^{k+1}}(0,v)\ne\vect{0}$, which implies 
$(x^{(k+1)}(0),z^{(k+1)}(0))\ne(0,0)$.
Thus, there exist functions $X(u),Z(u)$ on $I$
such that
\begin{equation}\label{eq:factor}
  (\dot{x}(u), \dot{z}(u) )= u^{k} (X(u),Z(u)), 
  \qquad
  (X(0),Z(0)) \ne (0,0).
\end{equation}

Third, we prove $z^{(k+1)}(0) \ne 0$.
Setting $\Lambda(u)=\sqrt{h^{2}X(u)^{2} + x(u)(X(u)^{2} + Z(u)^{2})}$,
we have that
$$
\nu(u,v)
= \frac{1}{\Lambda}
        (
        hX \sin v - xZ \cos v,
        -hX \cos v - xZ \sin v,
        xX
        )
$$
gives a unit normal vector field along $f$. 
Thus,
the limiting normal curvature $\kappa_\nu$
(cf.\ Definition \ref{def:generic})
of $f$ at $(0,v)$ is given by
$$
  \kappa_{\nu}
  =\left.\frac{f_{vv}\cdot \nu}
  {f_v\cdot f_v}\right|_{u=0}
  =\frac{x(0)^2}{(x(0)^2+h^2)\Lambda(0)}Z(0).
$$
Thus, $p=(0,v_0)$ is 
generic if and only if $Z(0) \ne 0$,
that is, $z^{(k+1)}(0) \ne 0$ holds.
\end{proof}

\begin{lemma}\label{lem:natural-coord}
Let $f:\Sigma\to \R^3$ be a generic helicoidal 
$n$-type edge,
where $n=k+1$ and $k$ is a positive integer.
For each singular point $p\in S(f)$,
there exist
an open interval $J$ including the origin,
a coordinate neighborhood
$(D,\varphi)$ giving a diffeomorphism
$\varphi=(s,\t):D\to J\times \R$,
and a positive-valued function $\RR(s)$ defined on $J$
such that 
the first fundamental form $\I$ of 
$f$ is given by \eqref{eq:metric-Bour-k},
that is
$\I = s^{2k}ds^2+\RR(s)^2d\t^2$ holds on $D$, 
and $\RR^{(i)}(0)=0$ holds for all $i$ with $1\le i \le k$,
where the $\RR^{(i)}$ means $d^{i}\RR/ds^{i}$.
\end{lemma}

\begin{proof}
Without loss of generality,
we may assume that
$f$ is given by $f_{\gamma,h}(u,v):I\times \R \to \R^3$
as in \eqref{cuspsur}.
We set 
$t(u,v):=v+\phi(u)$,
where 
$$
\phi(u):=h\int^u_{u_0}\frac{\dot{z}(\ivar)}{x(\ivar)^2+h^2}d\ivar.
$$
Setting
$\tilde{f}:I\times \R \to\R^3$
as
$\tilde{f}(u,t):=f_{\gamma,h}(u,t-\phi(u))$,
we have
$\tilde{f}_u\cdot\tilde{f}_t=0$.

For each $t\in \R$,
we set $\gamma : I\to \R^3$ as
$\gamma(u):=\tilde{f}(u,t)$.
By using Lemma \ref{lem2},
we may check that 
$$
\gamma^{(i)}(u_0)=\vect{0}
\quad
(i=1,\dots,k),\qquad
\gamma^{(k+1)}(u_0)\ne \vect{0}.
$$
Then, Fact \ref{fact:PBAL-riem} implies 
that there exists a coordinate change
$s=s(u)$ defined on an open neighborhood of $u_0\in I$
such that $\|d\gamma/ds\|=|s|^{k}$
and $s(u_0) = 0$.
We set $\hat{f}(s,t):=\tilde{f}(u(s),t)$,
where $u=u(s)$ is the inverse function of $s=s(u)$.
Since 
$\hat{f}_s=d\gamma/ds$,
we have
$\hat{f}_s\cdot\hat{f}_s=s^{2k}.$
Moreover, 
$\hat{f}_s=(du/ds)\tilde{f}_u$ yields
$\hat{f}_s\cdot\hat{f}_t=0$.
Since
$\tilde{f}_t\cdot \tilde{f}_t
=x(u(s))^{2}+h^2$, 
the first fundamental form $\I$
of $f$ is given by \eqref{eq:metric-Bour-k},
where we set
\begin{equation}\label{eq:U}
\RR(s):=\sqrt{(x\circ u)(s)^2+h^2}.
\end{equation}
Since $d^{i}x/du^{i}=0$ for all $i$ with $1\le i \le k$, we have 
$\RR^{(i)}(0)=0$,
which gives the desired result.
\end{proof}

\begin{theorem}[Bour-type representation formula for generic helicoidal $n$-type edges]
\label{thm:CE-3}
Fix a positive integer $k$
and set $n=k+1$.
Let $J$ be an open interval including the origin $0$,
and $\RR(s)$ be a positive valued smooth function 
on $J$ such that $\RR^{(i)}(0)=0$ for all $i$ with $1\le i \le k$.
Take a smooth function $\dR (s)$ on $J$
such that $\RR'(s)=s^{k}\dR (s)$.
Suppose that there exists a positive number $m>0$ and real number $h$
such that
\begin{equation}\label{eq:constants-CE}
\tag*{$(*)_{h,m}$}
\begin{split}
&\rho_{h,m}(s)
:=\sqrt{m^2 \RR(s)^2 - h^2 -m^4 \RR(s)^2 \dR (s)^2}\\
&\hspace{0.5cm}
\text{is a smooth function on $J$ which satisfies $\rho_{h,m}(0)\ne0$.}
\end{split}
\end{equation}
For $\epsilon_0, \epsilon_1, \epsilon_2 \in \{+1,-1\}$,
we set
$x(s)$, $z(s)$, $\theta(s,\t)$
as
\begin{equation}\label{eq:Bour}
\begin{split}
x(s)&= \epsilon_0\sqrt{m^2 \RR(s)^2-h^2},\\
z(s) &=\epsilon_2 m 
\int^s_0\frac{\ivar^k  \RR(\ivar)\rho_{h,m}(\ivar)}{m^2 \RR(\ivar)^2-h^2}d\ivar,\\
\theta(s,\t) &=
\frac1{m}\left(\epsilon_1 t
- \epsilon_2 h
\int^s_0\frac{\ivar^k \rho_{h,m}(\ivar)}{\RR(\ivar)(m^2 \RR(\ivar)^2-h^2)}d\ivar
\right), 
\end{split}
\end{equation}
respectively.
Then 
$\CC:J\times \R\to \R^3$ 
defined by
$$
 \CC(s,\t):= 
 \left( x(s) \cos \theta(s,\t),\, x(s) \sin \theta(s,\t),\, z(s)+h\theta(s,\t)\right)
$$
is a generic helicoidal $n$-type edge
whose first fundamental form
coincides with $\I = s^{2k}ds^2+\RR(s)^2d\t^2$ 
as in \eqref{eq:metric-Bour-k}.
Conversely,  
on a neighborhood of a singular point,
every generic helicoidal $n$-type edge
whose first fundamental form
coincides with $\I$ as in \eqref{eq:metric-Bour-k}
is given in this manner, up to a rigid motion of $\R^3$.
\end{theorem}

\begin{proof}
First, we prove the former assertion.
It holds that
$\CC_s(s,t)=s^k\overline{\CC}(s,t)$,
where we set
\begin{equation}\label{eq:bar-Psi}
\overline{\CC} = \frac{1}{m\RR x}
        \Bigl(
        m^3U^2V\cos \theta +\epsilon_2h \rho \sin \theta,~
        m^3U^2V\sin \theta -\epsilon_2h \rho \cos \theta,~
        \epsilon_2 \rho x
        \Bigr)
\end{equation}
and $\rho=\rho_{h,m}(s)$.
By a direct calculation, we obtain
\begin{equation}\label{eq:inner-product}
  \|\overline{\CC}\|^2=1,\qquad
  \overline{\CC}\cdot \CC_t=0,\qquad
  \|\CC_t\|^2=U^2.
\end{equation}
As $\CC_s=s^k\overline{\CC}$, 
we may conclude that
the first fundamental form of $\Psi$
coincides with $\I$ as in \eqref{eq:metric-Bour-k}.
Thus the singular set $S(f)$ coincides with $t$-axis.
As $\CC_s=s^k\overline{\CC}$, 
\eqref{eq:inner-product}
yield that $\xi:=\partial_t$ and $\eta:=\partial_s$ satisfy
the conditions \eqref{itm:cri1} and \eqref{itm:cri2}
with $n=k+1$.
By Fact \ref{fact:MSST}, 
every singular point $(0,t)\in J\times \R$
is a $(k+1)$-type edge singular point.

With respect to the genericity,
we calculate the limiting normal curvature 
$\kappa_\nu$
at a singular point $(0,t)\in J\times \R$
(cf.\ Definition \ref{def:generic}).
Since
$$
\nu (s,\t) = \frac{-\epsilon_2}{x}
        \Bigl(
        \epsilon_2 \rho \cos \theta - hm \dR \sin \theta,~
        \epsilon_2 \rho  \sin \theta + hm \dR \cos \theta,~
        -\epsilon_0 m\dR x
        \Bigr)
$$
gives a unit normal vector field along $\Psi$,
we have 
$$
\kappa_\nu(0,t)
=\frac{\CC_{tt}(0,t)\cdot \nu(0,t)}{U(0)^2}
=\frac{\rho(0)}{m^2\RR(0)^2}.
$$
Then \ref{eq:constants-CE} yields that 
$\kappa_\nu(0,t)\ne 0$,
and hence,
every singular point $(0,t)$ is generic.

Next, we prove the converse.
Let $f:\Sigma\to\R^3$ be a 
generic helicoidal $(k+1)$-type edge.
By Lemma \ref{lem:natural-coord},
for each singular point $p\in S(f)$,
there exist
an open interval $J$ including the origin,
a coordinate neighborhood
$(D,\varphi)$ giving a diffeomorphism
$\varphi=(s,\t):D\to J\times \R$,
and a positive-valued function $\RR(s)$ defined on $J$
such that 
the first fundamental form $\I$ of $f$ is written as 
\eqref{eq:metric-Bour-k} on $D$, 
and $\RR^{(i)}(0)=0$ holds for all $i$ with $1\le i \le k$.
We may set $\CC(s,\t):=f\circ \varphi^{-1}(s,t)$ as
$\CC(s,\t)= \left( x \cos \theta, x \sin \theta, z+h\theta\right)$,
where 
$x(s,\t)$, $z(s,\t)$ and $\theta(s,\t)$ are smooth functions on 
$J\times \R$.

According to Lemma $\ref{lem2}$, 
$x(0,0) \ne 0$,
and hence, 
$x(s,t) \ne 0$ holds
on a neighborhood of $(0,0)$.
Then, the first fundamental form $\I$ is given by
$$
  \I
  = dx^2 + \frac{x^2}{x^2+h^2}dz^{2}
  +(x^2+h^2)\left(d\theta+\frac{h}{x^2+h^2}dz\right)^2.
$$
By the construction method of $U(s)$
as in \eqref{eq:U}, we have that
\begin{align}
\label{eq:ds}
s^{2k}ds^2&=dx^2+\frac{x^2}{x^2+h^2}dz^2,\\
\label{eq:dt}
\RR(s)d\t&=\epsilon_1\sqrt{x^2+h^2}\left(d\theta+\frac{h}{x^2+h^2}dz\right)
\end{align}
for $\epsilon_1\in \{+1,-1\}$.
Substituting $dx=x_s ds+x_\t d\t$ and $dz=z_s ds+z_\t d\t$
into \eqref{eq:ds},
we have $x_\t = z_\t = 0$.
Hence, $x$ and $z$ do not depend on $t$,
namely, $x=x(s)$ and $z=z(s)$ hold.
Substituting $d\theta=\theta_{s}ds+\theta_{\t} d\t$
into \eqref{eq:dt}, we have
$$
  \theta_{\t}(s,\t) = \epsilon_1 \frac{\RR(s)}{\sqrt{x(s)^2+h^2}},
  \qquad
  \theta_s(s,\t) = -\frac{h}{x(s)^2+h^2} z'(s),
$$
where the prime means $d/ds$.
The second equation yields $\theta_{st}=0$,
and hence we may write $\theta(s,\t)=\alpha(\t)+\beta(s)$
for some functions $\alpha(\t)$, $\beta(s)$
of one variable.
Thus we have
\begin{equation}
\label{eq:vs-vt-2}
\frac{d}{d\t}\alpha(\t) = \epsilon_1 \frac{\RR(s)}{\sqrt{x(s)^2+h^2}},
\qquad
\frac{d}{ds}\beta(s) = -\frac{h}{x(s)^2+h^2} z'(s).
\end{equation}
The first equation of \eqref{eq:vs-vt-2}
implies that $(\RR(s)/\sqrt{x^2+h^2})_s=0$.
Hence there exists a positive constant $m>0$
such that
\begin{equation}
\label{eq:1/m}
\frac{\RR(s)}{\sqrt{x^2+h^2}} = \frac1{m}.
\end{equation}
Then, the first equation of \eqref{eq:vs-vt-2}
yields that
$d\alpha/d\t = \epsilon_1/m$.
By changing $t\mapsto t+{\rm constant}$ if necessary,
we may assume that $\alpha(0)=0$.
Hence, we have
$
\alpha(\t) = \epsilon_1 t /m.
$
Moreover, by \eqref{eq:1/m}, we have 
$x(s)=\epsilon_0\sqrt{m^2\RR(s)^2-h^2}$
for some $\epsilon_0\in \{+1,-1\}$.
Together with \eqref{eq:ds}, we may conclude that
$$
z'(s)^2 = \frac{m^2s^{2k}\RR(s)^2}{(m^2\RR(s)^2-h^2)^2} 
\left( m^2 \RR(s)^2 - h^2 -m^4 \RR(s)^2 \dR (s)^2 \right).
$$
By Lemma \ref{lem2}, 
there exists a smooth function $\bar{z}(s)$
such that $z'(s)=s^{k}\bar{z}(s)$ and $\bar{z}(0)\ne0$.
Restricting $J$ to a neighborhood of the origin $0$,
we may assume that $\bar{z}(s)\ne0$ on $J$.
Then  
$$
m^2 \RR(s)^2 - h^2 -m^4 \RR(s)^2 \dR (s)^2
=\frac{(m^2\RR(s)^2-h^2)^2}{m^2\RR(s)^2}\bar{z}(s)^2>0
$$
holds on $J$,
which implies \ref{eq:constants-CE}.
Then there exists $\epsilon_2\in \{+1,-1\}$
such that
$\frac{d}{ds}z(s) = {\epsilon_2 m s^{k}\,\RR(s) \rho_{h,m}(s)}/{x(s)^2}.$
Substituting this into 
the second equation of \eqref{eq:vs-vt-2},
we have
$$
\frac{d}{ds}\beta(s) = \frac{-\epsilon_2 h s^{k}\, \rho_{h,m}(s)}{m\RR(s)x(s)^2},
$$
which gives the desired result.
\end{proof}

We call $\{U,h,m,\epsilon_0, \epsilon_1, \epsilon_2\}$
the {\it data} of 
the generic helicoidal $n$-type edge
$\CC:J\times \R\to \R^3$
given in Theorem \ref{thm:CE-3}.
We also denote
$\CC$ by 
$\CC_{[U,h,m,\epsilon_0, \epsilon_1, \epsilon_2]}$
to emphasis the data.

In the proof of Theorem \ref{thm:CE-3},
we have proved the following.

\begin{lemma}
\label{lem:LNC}
The limiting normal curvature $\kappa_\nu$
of the generic helicoidal $n$-type edge
$\CC=\CC_{[U,h,m,\epsilon_0, \epsilon_1, \epsilon_2]}$
along the singular curve $\sigma(t)=(0,t)$
is given by 
$$
  \kappa_\nu = \frac{\rho_{h,m}(0)}{m^2\RR(0)^2},
$$
where $\rho_{h,m}$ is defined in \ref{eq:constants-CE}.
\end{lemma}

Similarly, the following holds.

\begin{lemma}
\label{lem:cuspidal-torsion}
The cusp-directional torsion $\kappa_t$ 
of the generic helicoidal $n$-type edge
$\CC=\CC_{[U,h,m,\epsilon_0, \epsilon_1, \epsilon_2]}$
along the singular curve $\sigma(t)=(0,t)$
is given by 
$$
  \kappa_t = \frac{h}{m^2 U(0)^2}.
$$
\end{lemma}

\begin{proof}
We use the same notation as in 
the proof of Theorem \ref{thm:CE-3}.
Since $\eta = \partial_s$ is a null vector field 
and $\xi = \partial_t$ is tangent to the singular set, 
the cusp-directional torsion $\kappa_t$ 
at a singular point $(0,t)$ is written as
\begin{equation}\label{eq:kt-1}
  \kappa_t 
  = \frac{\det(\CC_{t} , \CC_{s^{k+1}},\CC_{s^{k+1}t})}{\|\CC_{t} \times \CC_{s^{k+1}}\|^{2}}(0,t) 
  - \frac{(\CC_{t}\cdot \CC_{s^{k+1}}) \det(\CC_{t} , \CC_{s^{k+1}} , \CC_{tt}) 
  }{\|\CC_{t}\|^{2}\|\CC_{t} \times \CC_{s^{k+1}} \|^{2}}(0,t),
\end{equation}
where 
$\CC_{s^{k+1}t}=(\partial_s)^{k+1}\partial_t \Psi$.
As $\CC_s=s^k\overline{\CC}$, 
we have
$\CC_{s^{k+1}}(0,t)=k! \overline{\CC}(0,t)$,
where $\overline{\CC}$ is defined by
\eqref{eq:bar-Psi}.
Thus \eqref{eq:inner-product} yields that
$\CC_{s^{k+1}}(0,t)\cdot \CC_{t}(0,t)=0$,
and
$\|\CC_{t}(0,t) \times \CC_{s^{k+1}}(0,t)\|^{2}
  =(k!)^2U(0)^2.
$
By \eqref{eq:kt-1}, we have
\begin{equation}\label{eq:kt-2}
  \kappa_t 
  = \frac{1}{(k!)^2U(0)^2}
  \det(\CC_{t} , \CC_{s^{k+1}},\CC_{s^{k+1}t})(0,t).
\end{equation}
Since $\CC_{s^{k+1}t} (0,t)=k! \overline{\CC}_t(0,t)$
and 
$$
  \overline{\CC}_t
  = \frac{\epsilon_1}{m^2\RR x}
        \Bigl(
        -m^3U^2V\sin \theta +\epsilon_2h \rho \cos \theta,~
        m^3U^2V\cos \theta +\epsilon_2h \rho \sin \theta,~
        0\Bigr),
$$
we obtain
$
  \det(\CC_{t} , \CC_{s^{k+1}},\CC_{s^{k+1}t})(0,t)
  ={h(k!)^2}/{m^2}.
$
Together with \eqref{eq:kt-2},
we have $\kappa_t = h/(m^2 U(0)^2).$
\end{proof}

\begin{theorem}[Bour-type isometric deformation theorem]
\label{thm:singular-Bour}
Let $f:\Sigma\to\R^3$ be a generic helicoidal $n$-type edge,
and $p\in \Sigma$ be a singular point.
Then, there exist
an open neighborhood $D$ of $p$,
and a $2$-parameter family of 
generic helicoidal $n$-type edges
$\{f_{h,m} : D \to \R^3 \}_{h,m}$
which gives a non-trivial isometric deformation 
of $f$.
Conversely,
every non-trivial isometric deformation 
of $f$ among generic helicoidal $n$-type edges
is given by $\{f_{h,m} \}_{h,m}$.
\end{theorem}

\begin{proof}
Let $n=k+1$ and $h_0$ be the pitch of $f$.
By Lemma \ref{lem:natural-coord},
there exist
an open interval $J$ including the origin,
a coordinate neighborhood
$(D,\varphi)$ giving a diffeomorphism
$\varphi=(s,\t):U\to J\times \R$,
and a positive-valued function $\RR(s)$ defined on $J$
such that $\RR'(0) = \cdots =\RR^{(k)}(0) =0$ and
the first fundamental form $\I$ of $f$ is given by 
\eqref{eq:metric-Bour-k}.

By an isometry of $\R^3$ and 
a parameter translation $u\mapsto u+\text{constant}$,
we may assume that 
the original frontal $f(s,\t)$ is given by
$\CC_{[U,h_0,1,\epsilon_0,\epsilon_1,\epsilon_2]}(s,\t)$
for some $\epsilon_0,\epsilon_1,\epsilon_2\in \{+1,-1\}$.
In particular, \ref{eq:constants-CE}
holds for $h=h_0$, $m=1$.
Then, there exist constants
$h,m$ such that $h$ (resp.\ $m$) 
is sufficiently close to $h_0$ (resp.\ $1$),
and \ref{eq:constants-CE} holds.
Hence, 
setting $f_{h,m}:=\CC_{[U,h,m,\epsilon_0,\epsilon_1,\epsilon_2]}$,
we have that
the family $\{f_{h,m}\}_{h,m}$
gives an isometric deformation of 
$f\,(=\CC_{[U,h_0,1,\epsilon_0,\epsilon_1,\epsilon_2]})$.

To prove the non-triviality of the isometric deformation
$\{f_{h,m}\}_{h,m}$,
we set $\psi(h,m):=(\kappa_\nu,\kappa_t)$.
By Lemmas \ref{lem:LNC} and \ref{lem:cuspidal-torsion},
we have
$
  \psi(h,m)
  = \left(
      \rho_{h,m}(0),~
      h
     \right)/(m^2 U(0)^2).
$
Since the Jacobian determinant of $\psi$, 
$$
   \det
   \begin{pmatrix}
   \partial \kappa_\nu/\partial h & \partial \kappa_\nu/\partial m\\
   \partial \kappa_t/\partial h & \partial \kappa_t/\partial m
   \end{pmatrix}
   = \frac{1}{m^3 \rho_{h,m} (0) U(0)^2},
$$
is not zero,
we have that $\psi$ is a local diffeomorphism,
and hence,
the family $\{f_{h,m}\}_{h,m}$
gives a non-trivial isometric deformation of $f$.
The converse assertion follows from 
Theorem \ref{thm:CE-3}.
\end{proof}

Let $\CC_{h_0}:=\CC_{[U,h_0,m,\epsilon_0,\epsilon_1,\epsilon_2]}$
be the generic helicoidal $n$-type edge given in Theorem \ref{thm:CE-3}.
In particular, $(*)_{h_0,m}$ holds.
Since 
$$
\rho_{0,m}\ge \rho_{h,m}\ge \rho_{h_0,m}
\qquad
(0\le h \le h_0)
$$
holds,
we have $(*)_{h,m}$ for each $h\in [0,h_0]$.
Hence $\{\CC_{h}\}_{h\in [0,h_0]}$
defined by
$\CC_{h}:=\CC_{[U,h,m,\epsilon_0,\epsilon_1,\epsilon_2]}$
gives an isometric deformation
between $\CC_{h_0}$ and $\CC_0$.
Since $\CC_0$ is of pitch $0$,
we have the following.

\begin{corollary}
\label{cor:revolution}
For each generic helicoidal $n$-type edge,
there exists a family of generic helicoidal $n$-type edges
which gives a local isometric deformation 
to a generic $n$-type edge of revolution.
\end{corollary}

\begin{example}\label{ex:GHF}
Let $k$ be a positive integer, and set
$$
\RR(s):=
1+\int_{0}^s \zeta^k\,\sin \zeta\,d\zeta.
$$
We consider a generic helicoidal $n$-type edge
$(n=k+1)$
given by our Bour-type representation formula.
Since $\RR'(s)=s^k\,\sin s$,
we have $\dR (s)=\sin s$.
Substituting
$m=1$,
$\epsilon_0=\epsilon_1=1$
and 
$\epsilon_2=-1$
into $\CC_{[U,h,m,\epsilon_0,\epsilon_1,\epsilon_2]}(s,\t)$ 
given in Theorem \ref{thm:CE-3},
we obtain 
a $1$-parameter family of
generic helicoidal $n$-type edges
$\CC_h(s,\t):=\CC_{[U,h,1,+,+,-]}(s,\t)$.
This family $\{\CC_h\}_h$ has the common first fundamental form
$\I=s^{2k}\,ds^2+U(s)^2\,dt^2$.
See Figures \ref{fig:Bour-23CE} and \ref{fig:Bour-CE-3,4cusp} 
for $k=1,2$, respectively.
In the case of $k=1$ 
(resp. $k=2$), 
every singular point of 
this $\CC_h(s,\t)$
is in fact cuspidal edge
(resp. $4/3$-cuspidal edge), 
see Proposition \ref{prop:criterion}.
A part of $\CC_0$ for $k=1$ is shown in Figure \ref{fig:CE}, right.
\end{example}

\begin{figure}[htb]
\begin{center}
 \begin{tabular}{{c@{\hspace{6mm}}c@{\hspace{6mm}}c}}
  \resizebox{3.5cm}{!}{\includegraphics{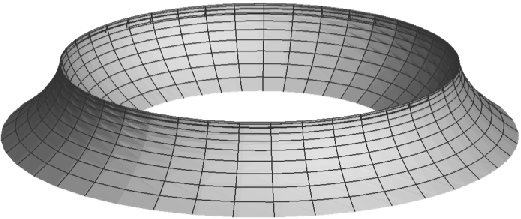}}&
  \resizebox{3.5cm}{!}{\includegraphics{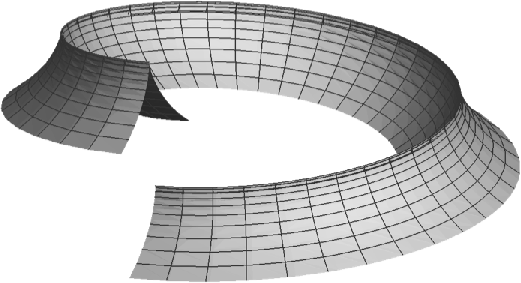}}&
  \resizebox{3.5cm}{!}{\includegraphics{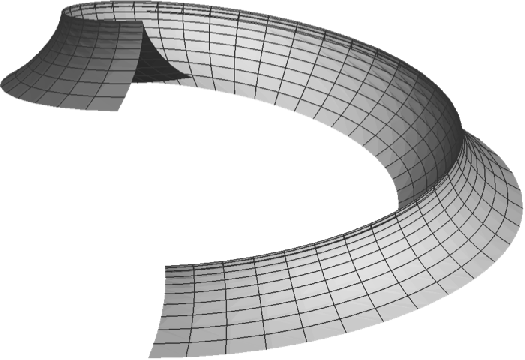}}
 \end{tabular}
\end{center}
\caption{
The generic helicoidal cuspidal edges $\CC_h$
in Example \ref{ex:GHF}.
The left (resp.\ center, right) figure shows
the graphic of $h=0$
(resp. $h=0.1$, $0.2$).}
\label{fig:Bour-23CE}
\end{figure}

\begin{figure}[htb]
\begin{center}
 \begin{tabular}{{c@{\hspace{6mm}}c@{\hspace{6mm}}c}}
  \resizebox{3.5cm}{!}{\includegraphics{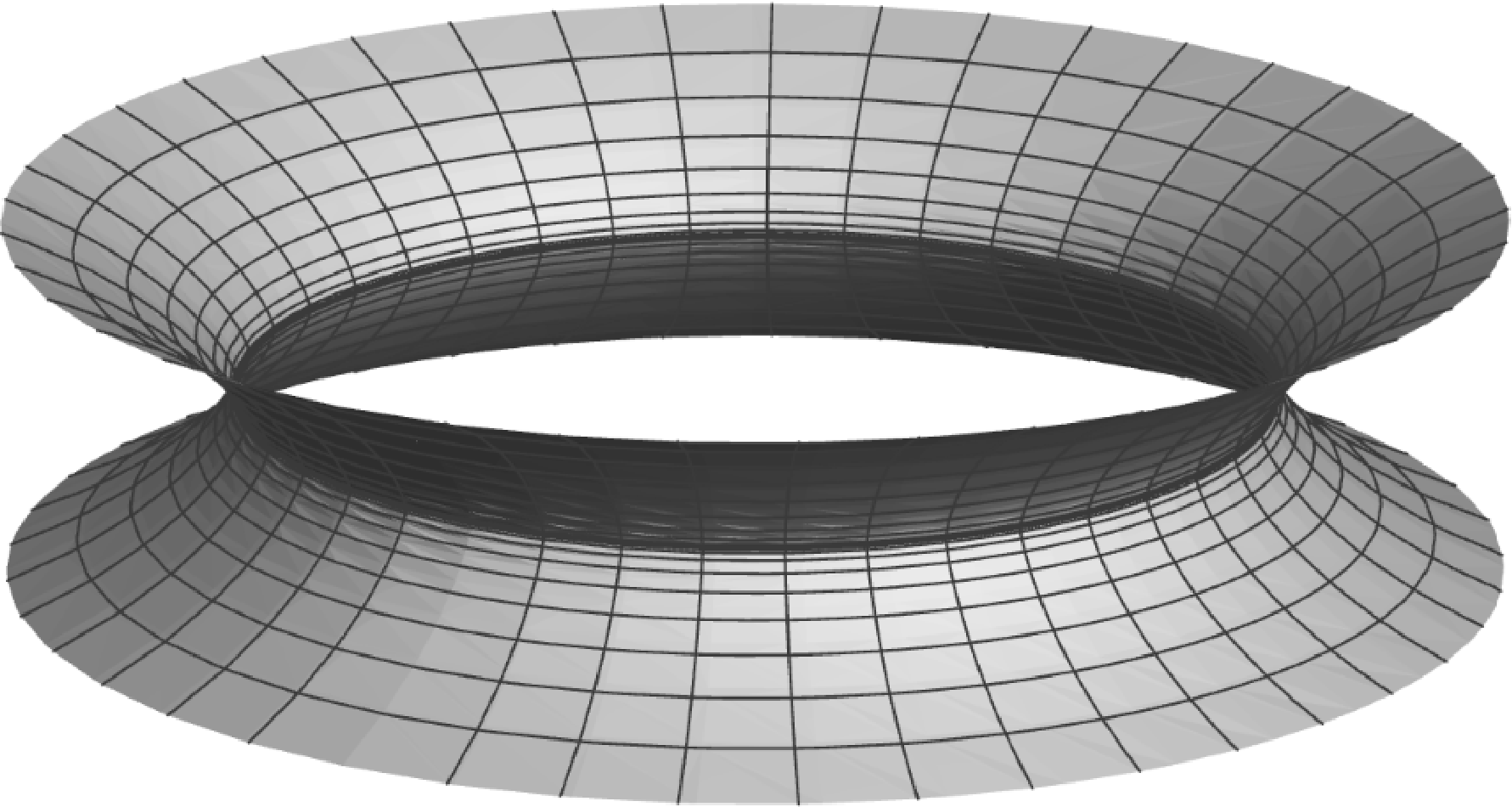}}&
  \resizebox{3.5cm}{!}{\includegraphics{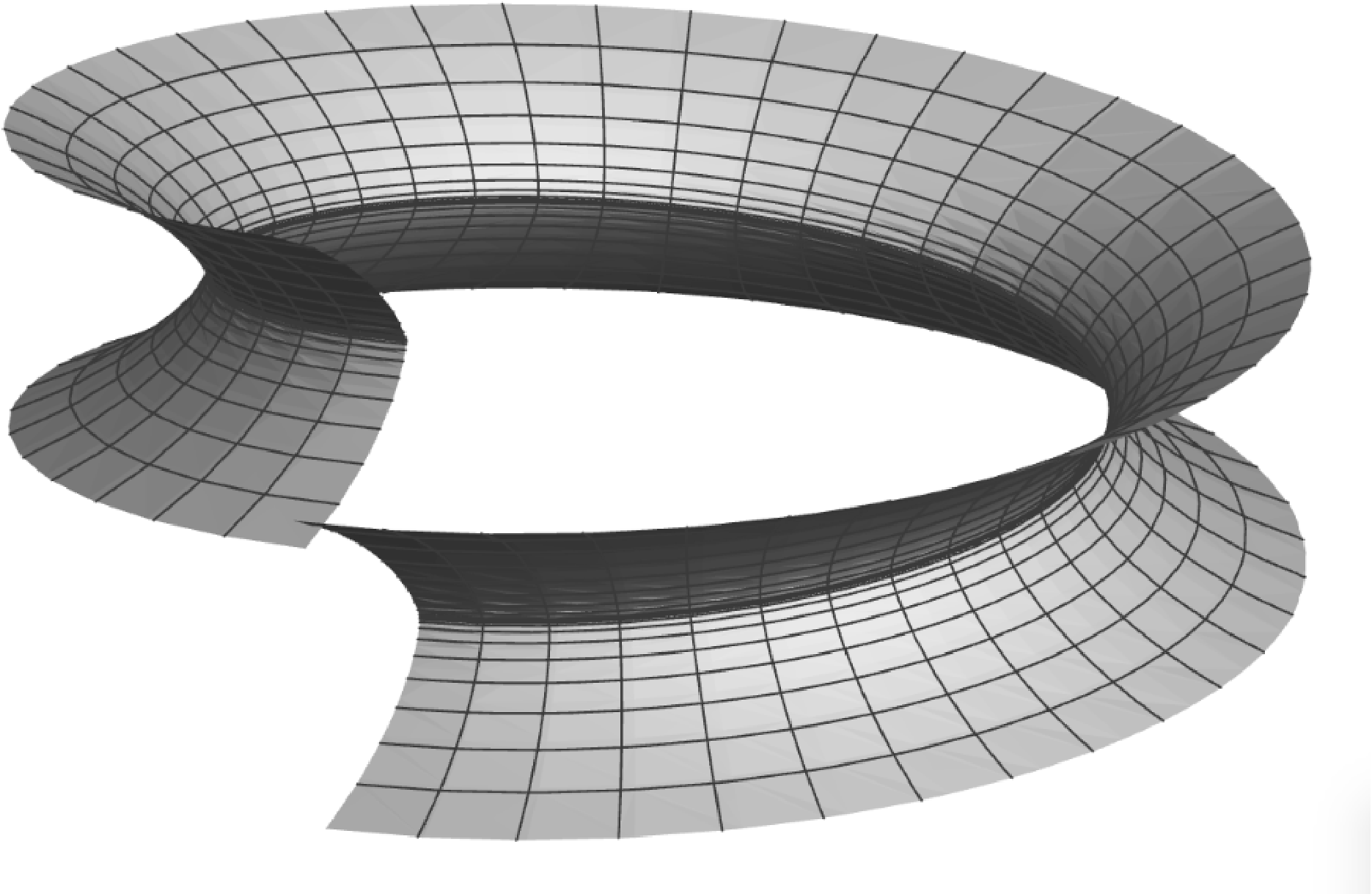}}&
  \resizebox{3.5cm}{!}{\includegraphics{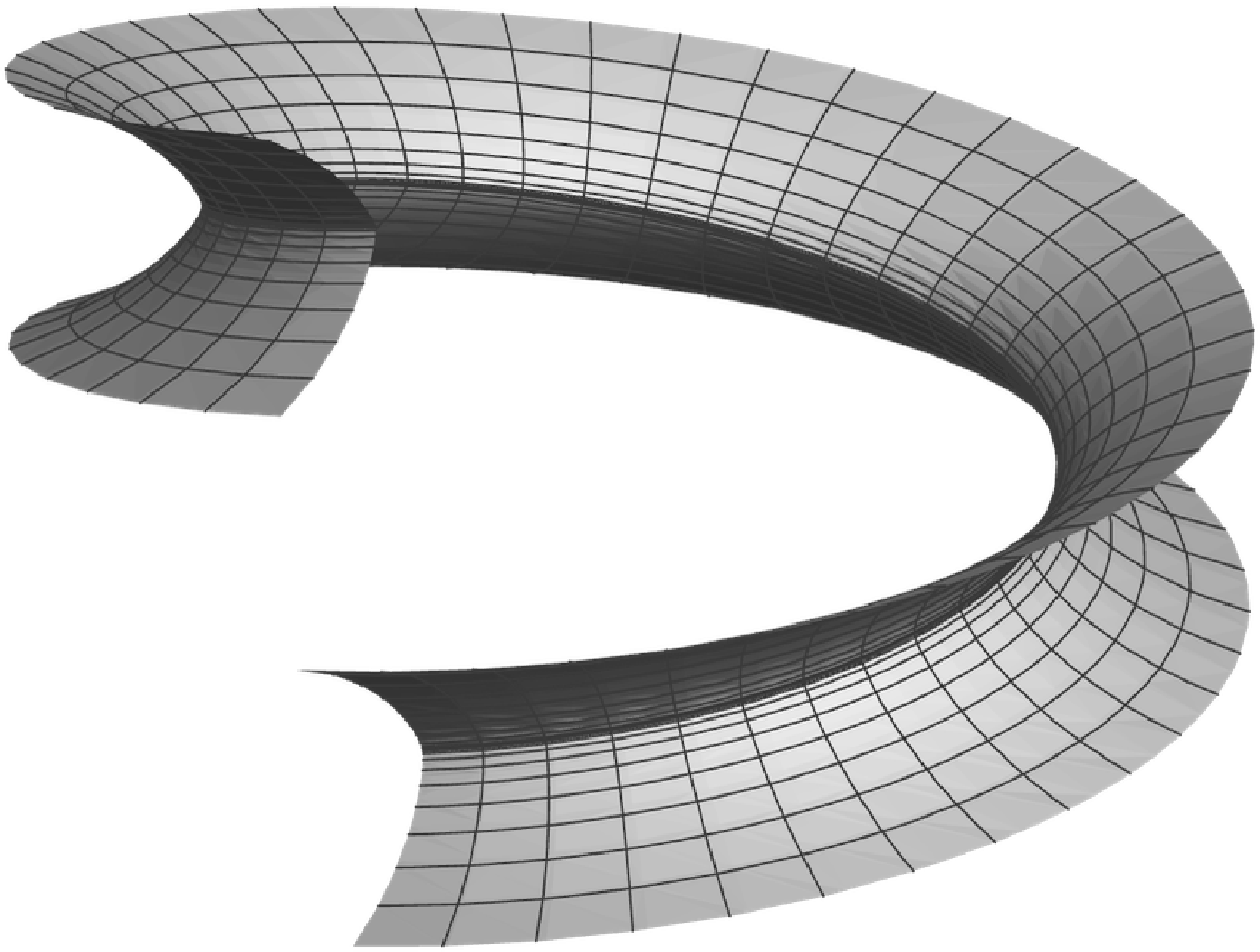}}
 \end{tabular}
\end{center}
\caption{The generic helicoidal $4/3$-cuspidal edges $\CC_h$
in Example \ref{ex:GHF}.
The left (resp.\ center, right) figure shows
the graphic of $h=0$
(resp. $h=0.1$, $0.2$).}
\label{fig:Bour-CE-3,4cusp}
\end{figure}

\begin{remark}[Isomers of cuspidal edges]
\label{rem:isomer}
Let $f:\Sigma\to \R^3$ be a real analytic front 
having a generic cuspidal edge $p\in \Sigma$.
Denote by $\I$ 
the first fundamental form of $f$.
Take a singular curve $\sigma(\t)$ passing through
$p=\sigma(0)$,
and set $\hat{\sigma}(t):=f\circ \sigma(\t)$.
Let $C:={\rm Image}(\hat{\sigma})$ 
be the image of $\hat{\sigma}$,
which is called the {\it singular value set} of $f$.
In \cite{HNSUY1} (see also \cite{NUY, HNUY}),
it is proved that 
there exist an open neighborhood $D$ of $p$
and real analytic fronts 
$\check{f}, f_*, \check{f}_*$ defined on $D$
such that 
\begin{itemize}
\item the first fundamental forms of 
$\check{f}, f_*, \check{f}_*$ coincide with $\I$,
and
\item
the singular value sets of 
$\check{f}, f_*, \check{f}_*$ are subset of $C$.
\end{itemize}
The orientation of 
the singular value set $C$ of $f$ 
(resp.\ $f_*$)
coincides that of $\check{f}$
(resp.\ $\check{f}_*$).
Moreover, 
the singular value set of $f$ 
has the opposite orientation 
to that of $f_*$.
Hence, 
$\check{f}$
(resp.\ $\check{f}_*$)
is called the {\it isometric dual}
of $f$
(resp.\ $f_*$).
In \cite{HNSUY1},
the existence of 
$\check{f}, f_*, \check{f}_*$
is proved by using the Cauchy-Kowalevski theorem.
Hence, it is not easy to construct 
the isometric dual $\check{f}$ of 
a given real analytic cuspidal edge $f$
explicitly, in general.
However, if $f$ has 
a {\it symmetry}, 
then $\check{f}$ is given in an explicit way
(see \cite{HNSUY1}, for details).

In the case of generic helicoidal cuspidal edges,
we can construct the isometric duals explicitly.
In fact, generic helicoidal cuspidal edges
have a symmetry \cite[Example 5.4]{HNSUY1}.
For $f:=\CC_{[U,h,m,\epsilon_0,+,+]}$ 
defined as in Theorem \ref{thm:CE-3},
the isometric dual $\check{f}$ of $f$
is given by
$\CC_{[U,h,m,\epsilon_0,+,-]}$.
Moreover, $f_*$ and $\check{f}_*$ are given by
$\CC_{[U,h,m,\epsilon_0,-,+]}$
$\CC_{[U,h,m,\epsilon_0,-,-]}$.
In other words, the four generic helicoidal cuspidal edges
$$
  \CC_{[U,h,m,\epsilon_0,+,+]},\quad
  \CC_{[U,h,m,\epsilon_0,+,-]},\quad
  \CC_{[U,h,m,\epsilon_0,-,+]},\quad
  \CC_{[U,h,m,\epsilon_0,-,-]}
$$
have the common first fundamental form
and the common singular value set.

Let $f_1, f_2$ be cuspidal edges
having
the common first fundamental form
and the common singular value set.
If the image of $f_1$
does not coincide with that of $f_2$
as set germs at singular points,
then $f_2$ is called the {\it isomer} of $f_1$
(\cite{NUY, HNSUY1, HNSUY2}).
In our cuspidal edges, 
$\CC_{[U,h,m,\epsilon_0,+,+]}$
is an isomer of
$\CC_{[U,h,m,\epsilon_0,+,-]}$ (cf.\ Figure \ref{fig:Bour-isomer}).

\begin{figure}[htb]
\begin{center}
 \begin{tabular}{{c}}
  \resizebox{3.5cm}{!}{\includegraphics{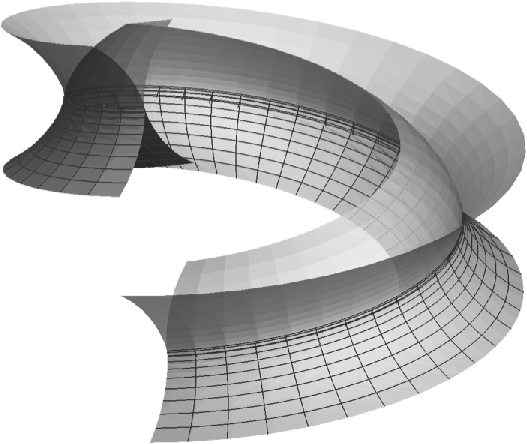}}
 \end{tabular}
\end{center}
\caption{The generic helicoidal cuspidal edge $\CC_h$
in Example \ref{ex:GHF} (meshed one), 
and the isomer of $\CC_h$ (the one without mesh)
in the case of $h=0.2$.}
\label{fig:Bour-isomer}
\end{figure}

In \cite{HNSUY2}, 
it is proved that there exists infinitely many
isomers of generic real analytic cuspidal edges
along a knot.
\end{remark}

\section{Applications of Bour-type theorem}
\label{sec:application}

In this section, as applications of our Bour-type theorem
(Theorems \ref{thm:CE-3} and \ref{thm:singular-Bour}),
we prove the existence of 
isometric deformations of $r/n$-cuspidal edges
which are invariant under helicoidal motions
(Corollaries \ref{cor:Bour-CE} and \ref{cor:Bour-45CE}),
and show that 
several invariants of generic helicoidal finite type edges 
are extrinsic invariants
(Corollaries \ref{cor:nu-c} and \ref{cor:kappa-c}).

\subsection{Isometric deformations of $r/n$-cuspidal edges}
\label{sec:nr-ce}

Throughout this subsection, $Q$ denotes 
the following set:
$$
Q=\left\{
3/2, 5/2, 7/2, 4/3, 5/3, 5/4
\right\}.
$$
For coprime positive integers $n,r$ satisfying
$r/n\in Q$,
a smooth surface $F_{n,r}:\R^2\to \R^3$ defined by
$$
F_{n,r}(u,v):=(u^{n},u^{r},v)\qquad
(u,v\in \R)
$$
is called 
the {\it standard $r/n$-cuspidal edge},
or the {\it standard $(n,r)$-cuspidal edge}.
The standard cuspidal edge $f_C$
defined in the introduction coincides with $F_{2,3}$.
Also, we set
$F_{4,5,7}, F_{4,5,-7}:\R^2\to \R^3$ as
$$
F_{4,5,\pm 7}(u,v):=(u^{4},u^{5}\pm u^{7}, v)\qquad
(u,v\in \R).
$$
We call
$F_{4,5,\pm 7}$
the {\it standard $(4,5,\pm7)$-cuspidal edge}.

\begin{definition}
Let $f: \Sigma\to \R^3$ be a smooth surface.
A singular point $p\in S(f)$
is said to be {\it $r/n$-cuspidal edge\/}, 
or {\it $(n,r)$-cuspidal edge\/}, if 
the map-germ $f : (\Sigma,p)\to (\R^3,f(p))$
is $\mathcal{A}$-equivalent to 
the map-germ $F_{n,r} : (\R^2,0)\to (\R^3,0)$.
Similarly,
a singular point $p\in S(f)$
is said to be {\it $(4,5,\pm7)$-cuspidal edge\/}, if 
the map-germ $f : (\Sigma,p)\to (\R^3,f(p))$
is $\mathcal{A}$-equivalent to 
the map-germ $F_{4,5,\pm7} : (\R^2,0)\to (\R^3,0)$.
\end{definition}

We remark that an $r/n$-cuspidal edge singular point
is $n$-type edge,
and a $(4,5,\pm7)$-cuspidal edge singular point
is $4$-type edge.
The purpose of this subsection is to clarify how these 
singularities are deformed under the Bour-type isometric deformation.
First, we prepare the criteria for singularities of plane curves.

For coprime positive integers $n,r$ satisfying
$r/n\in Q$, we set
$$
\Gamma_{n,r}(u):=(u^{n},u^{r})\qquad
(u\in \R).
$$
The map $\Gamma_{n,r}:\R\to \R^2$ 
is called the {\it standard $r/n$-cusp},
or the {\it standard $(n,r)$-cusp}.
The multiplicity of $\Gamma_{n,r}$ at $u=0$
is given by $n$.
Moreover we set
$$
\Gamma_{4,5,\pm7}(u):=(u^{4},u^{5}\pm u^{7})\qquad
(u\in \R).
$$
The map $\Gamma_{4,5,\pm7}:\R\to \R^2$ 
is called the {\it standard $(4,5,\pm7)$-cusp}.

\begin{definition}
\label{def:nr-cusp}
Let $\gamma: I\to \R^2$ be a smooth curve.
For $r/n\in Q$,
a singular point $u_0\in I$ of $\gamma$ 
is said to be {\it $r/n$-cusp\/},
or {\it $(n,r)$-cusp\/}, if 
the map-germ $\gamma : (I,u_0)\to (\R^2,\gamma(u_0))$
is $\mathcal{A}$-equivalent to 
the map-germ $\Gamma_{n,r} : (\R,0)\to (\R^2,0)$.
Similarly,
a singular point $u_0\in I$ of $\gamma$ 
is said to be {\it $(4,5,\pm7)$-cusp\/}, if 
the map-germ $\gamma : (I,u_0)\to (\R^2,\gamma(u_0))$
is $\mathcal{A}$-equivalent to 
the map-germ $\Gamma_{4,5,\pm7} : (\R,0)\to (\R^2,0)$.
\end{definition}

The following criteria are known
(cf.\ \cite{Bruce-Gaffney}, \cite{porteous}).

\begin{fact}
\label{criteria_cusp}
Let $\gamma:I\to \R^2$ be a smooth curve
having a singular point $u_0\in I$.
Then,
\begin{enumerate}
\item\label{itm:23}
$\gamma$ has $3/2$-cusp at $u_0$
if and only if 
$\det(\gamma'',\gamma''')(u_0)\ne0.$
\item\label{itm:25} 
$\gamma$ has $5/2$-cusp at $u_0$
if and only if 
there exists $c_1\in \R$
such that
$$
\gamma'''(u_0)=c_1\,\gamma''(u_0),\quad
\det(\gamma'',3\gamma^{(5)}-10c_1\gamma^{(4)})(u_0)\ne0.
$$
\item\label{itm:34} 
$\gamma$ has $4/3$-cusp at $u_0$
if and only if $\gamma''(u_0)=0$ and
$\det(\gamma''',\gamma^{(4)})(u_0)\ne0.$
\item\label{itm:35}
$\gamma$ has $5/3$-cusp at $u_0$
if and only if 
$\gamma''(u_0)=0$ and
$$
\det(\gamma''',\gamma^{(4)})(u_0)=0,\quad
\det(\gamma''',\gamma^{(5)})(u_0)\ne0.
$$
\end{enumerate}
\end{fact}

In the case of $7/2$-cusp,
we obtain a similar criterion,
which is proved in Appendix \ref{sec:7/2}.
This criterion is based on the work by Takagi \cite{Takagi}.

\begin{theorem}\label{thm:7/2}
Let $\gamma:I\to \R^2$ be a smooth curve
having a singular point $u_0\in I$.
Then, $\gamma$ has $7/2$-cusp at $u_0$ if and only if 
there exist real numbers $c_1,c_2  \in \R$
such that $\ga'''(u_0) = c_1\ga''(u_0)$,
$3\ga^{(5)}(u_0)-10c_1 \ga^{(4)}(u_0)=3c_2\ga''(u_0)$,
and 
$$
\det\left(\ga'', \ga^{(7)}-7c_1 \ga^{(6)}-\left( 7c_2-\frac{70}{3}c_1^3\right)\ga^{(4)} \right)(u_0) \ne 0.
$$
\end{theorem}

Recently, Matsushita \cite{Matsushita}
proved the following.
\begin{fact}[\cite{Matsushita}]
\label{fact:Matsushita}
Let $\gamma:I\to \R^2$ be a smooth curve
satisfying 
$\gamma'=\gamma''=\gamma'''=\vect{0}$
at $u_0\in I$.
We set 
\begin{multline*}
\delta
=-77\det(\gamma^{(4)}, \gamma^{(6)})(u_0)^2
+105 \det(\gamma^{(4)}, \gamma^{(5)})(u_0)\det(\gamma^{(5)}, \gamma^{(6)})(u_0)\\ 
+60\det(\gamma^{(4)}, \gamma^{(5)})(u_0)
\det(\gamma^{(4)}, \gamma^{(7)})(u_0).
\end{multline*}
Then,
\begin{enumerate}
\item\label{itm:45}
$\gamma$ has $5/4$-cusp at $u_0$
if and only if 
$\det(\gamma^{(4)},\gamma^{(5)})(u_0)\ne0$
and $\delta=0$.
\item\label{itm:457} 
$\gamma$ has $(4,5,\pm7)$-cusp at $u_0$
if and only if 
$\det(\gamma^{(4)},\gamma^{(5)})(u_0)\ne0$
and $\pm\delta>0$.
\end{enumerate}
\end{fact}

Let $u_0\in I$ be a point and
$\gamma=(x,z):I\to \R^2$ be a curve
with $x\ne0$ on $I$.
For $r/n\in Q$,
the curve $\gamma$ has $r/n$-cusp at $u_0$
if and only if 
$f_{\gamma,h}$ has $r/n$-cuspidal edge at $(u_0,v_0)$
for each real number $v_0$.
This is verified by 
$$
  \Phi \circ F_{n,r}(u,v)
  = ((u^n+1)\cos v,(u^n+1)\sin v,u^r)
  = f_{\gamma,h}(u,v),
$$
where $\gamma(u)=(u^n+1,u^r)$ and 
$\Phi : (x,y,z)\mapsto ((x+1)\cos y,(x+1)\sin y,z)$
is a diffeomorphism for $x\ne -1$
(cf.\ \cite{MSST19}).
Similarly, 
the curve $\gamma$ has $(4,5,\pm7)$-cusp at $u_0$
if and only if 
$f_{\gamma,h}$ has $(4,5,\pm7)$-cuspidal edge at $(u_0,v_0)$
for each real number $v_0$.

Now, we prove the following.

\begin{proposition}
\label{prop:criterion}
Let 
$\CC(s,\t)$
be the generic helicoidal $n$-type edge
given in Theorem \ref{thm:CE-3}.
A singular point
$p=(0,t)\in S(\CC)$ is 
\begin{enumerate}
\item
$3/2$-cuspidal edge
if and only if $n=2$ and $\RR'''(0) \neq 0$.
\item
$5/2$-cuspidal edge
if and only if
$n=2$, $\RR'''(0) = 0$ and $\RR^{(5)}(0) \neq 0$.
\item
$7/2$-cuspidal edge
if and only if
$n=2$, $\RR'''(0) =\RR^{(5)}(0) = 0$ and $\RR^{(7)}(0) \ne 0$.
\item
$4/3$-cuspidal edge
if and only if
$n=3$ and $\RR^{(4)}(0) \ne 0$.
\item
$5/3$-cuspidal edge
if and only if
$n=3$, $\RR^{(4)}(0) = 0$ and $\RR^{(5)}(0) \neq 0$.
\item
$5/4$-cuspidal edge
if and only if
$n=4$, $\RR^{(5)}(0) \ne 0$ and $W=0$.
\item
$(4,5,\pm7)$-cuspidal edge
if and only if
$n=4$, $\RR^{(5)}(0) \ne 0$ and $\pm W>0$.
\end{enumerate}
Here we set 
\begin{multline*}
W=
2304  \left(60 U^{(5)}(0) U^{(7)}(0)-77 U^{(6)}(0)^2\right) \rho^4 \\
+25 m^4  U(0)^2 U^{(5)}(0)^2 \left(101 U^{(5)}(0)^2-16 U^{(4)}(0) U^{(6)}(0)\right) \rho^2 \\
+1675 m^4 U(0)^2 U^{(5)}(0)^4 \left(m^2 U(0)^2-h^2\right),
\end{multline*}
where $\rho=\rho_{h,m}(0)$.
\end{proposition}

\begin{proof}
We let $n=k+1$.
Setting $v:=\theta(s,t)$,
we have that $(s,t)\mapsto (s,v)$ is a local diffeomorphism.
In this coordinate $(s,v)$, 
the generic helicoidal $(k+1)$-type edge
$\CC$ is written as
$\left( x(s) \cos v,\, x(s) \sin v,\, z(s)+hv\right)$.
Then,
$\CC$ has $r/n$-cuspidal edge at $(0,t)$
if and only if $\gamma(s)=(x(s),z(s))$
has $r/n$-cusp at $s=0$.
Dividing $x$ and $z$ by some constants, 
we may set
$$
  \gamma(s)
  = \left(  
  \lambda(s),~
  \int^s_0
  \frac{\ivar^k  \RR(\ivar)\sqrt{\lambda(\ivar)^2-m^2\RR(\ivar)^2\dR(\ivar)^2}
  }{\lambda(\ivar)^2}d\ivar
  \right),
$$
where $\lambda(s)=\sqrt{\RR(s)^2-\alpha}$ and $\alpha=h^2/m^2$.
Then
$$
  \gamma'(s)
  = \frac{s^k}{\lambda(s)^2}\left(
  \RR(s)\dR(s)\lambda(s),~  
  \RR(s)\sqrt{\lambda(s)^2-m^2\RR(s)^2\dR(s)^2}
  \right)
$$
holds.
We set 
$\Delta=\sqrt{U(0)^2-\alpha} \sqrt{U(0)^2-m^2 U(0)^2V(0)^2-\alpha}$.

Let us consider the case $k=1$.
Then $\gamma'(0)=0$.
By a direct calculation, we have
$
\det(\gamma''(0),\gamma'''(0))=
-U(0)^2 U'''(0)/\Delta.
$
Hence 
Fact \ref{criteria_cusp} \ref{itm:23} yields that 
$\gamma(s)$ has cusp at $s=0$
if and only if $U'''(0)\ne0$,
which yields (1).
When $U'''(0)=0$, we have $\gamma'''(0)=0$,
hence the constant $c_1$
in Fact \ref{criteria_cusp} \ref{itm:25} is zero.
By a direct calculation, we have
$
\det(\gamma''(0),\gamma^{(5)}(0))=
-U(0)^2 U^{(5)}(0)/\Delta.
$
Hence 
Fact \ref{criteria_cusp} \ref{itm:25} yields that 
$\gamma(s)$ has $5/2$-cusp at $s=0$
if and only if $U'''(0)=0$ and $U^{(5)}(0)\ne0$,
which yields (2).
When $U'''(0)=U^{(5)}(0)=0$, 
we have $\gamma'''(0)=\gamma^{(5)}(0)=0$,
hence the constants $c_1$, $c_2$
in Theorem \ref{thm:7/2} are both zero.
By a direct calculation, we have
$
\det(\gamma''(0),\gamma^{(7)}(0))=
-U(0)^2 U^{(7)}(0)/\Delta.
$
Hence 
Theorem \ref{thm:7/2} yields that 
$\gamma(s)$ has $7/2$-cusp at $s=0$
if and only if $U'''(0)=U^{(5)}(0)=0$ and $U^{(7)}(0)\ne0$,
which yields (3).

As in the case where $k=1$,
the case where $k$ is $2$ or $3$ can also be shown using 
Facts \ref{criteria_cusp} and \ref{fact:Matsushita}.
Hence we obtain (4)--(7).
\end{proof}

\begin{corollary}
\label{cor:Bour-CE}
We let $r/n\in Q\setminus\{5/4\}$.
Every generic helicoidal $r/n$-cuspidal edge
admits a local non-trivial isometric deformation
given by a $2$-parameter family of generic helicoidal $r/n$-cuspidal edges.
\end{corollary}

\begin{proof}
We have already constructed 
such the isometric deformation $\{f_{h,m}\}_{h,m}$
in Theorem \ref{thm:singular-Bour}.
Every generic helicoidal $n$-type edge
in the family $\{f_{h,m}\}_{h,m}$ has
the common first fundamental form
$\I$ as in \eqref{eq:metric-Bour-k}.
The condition of singular points
to be $r/n$-cuspidal edge
is given in Proposition \ref{prop:criterion} (1)--(5),
which depends only on $\I$.
Hence, singular points of an element in $\{f_{h,m}\}_{h,m}$
have the same singularity types,
which proves the assertion.
\end{proof}

In the cases of 
$5/4$-cuspidal edges and 
$(4,5\pm7)$-cuspidal edges,
the criteria given in Proposition \ref{prop:criterion} (6) and (7)
depend on $h$ and $m$.
The conditions $W>0$ or $W<0$
are preserved under small perturbations of $h,m$.
Hence, we have the following.

\begin{corollary}
\label{cor:Bour-45CE}
Every generic helicoidal $(4,5\pm7)$-cuspidal edge
admits a local non-trivial isometric deformation
given by a $2$-parameter family of 
generic helicoidal $(4,5\pm7)$-cuspidal edges.
\end{corollary}

In the case of $5/4$-cuspidal edge,
since both $\partial W/\partial h$
and $\partial W/\partial m$ do not vanish,
our Bour-type isometric deformation changes 
the singularity type.

In general, a frontal $f: \Sigma \to \R^3$
is called {\it generic}, if it has non-empty singular set $S(f)$,
every singular point is rank one ,
and the limiting normal curvature does not vanish
on $S(f)$. 
Let $f_1, f_2 : \Sigma \to \R^3$
be two generic frontals 
so that $f_1$ and $f_2$ are isometric.
Then $S(f_1)=S(f_2)$ holds.
It is known that,
if $f_1$ has a cuspidal edge singular point $p\in S(f_1)$,
then $f_2$ also has cuspidal edge at $p\in S(f_2)$
(\cite[Proposition 4]{HNUY}).
Such a property is known to hold for swallowtail,
cuspidal cross cap, and $5/2$-cuspidal edges
(\cite[Proposition 4]{HNUY} and \cite[Lemma 5.4]{Honda-Saji}).
Until now, we did not know whether isometric generic frontals are 
$\mathcal{A}$-equivalent or not. 
Our example of generic helicoidal $5/4$-cuspidal edges 
provides a counterexample to this.

\begin{corollary}
\label{cor:rigidity}
There exist generic frontals 
$f_1, f_2 : \Sigma \to \R^3$
such that $f_1$ is isometric to $f_2$,
but $f_1$ is not $\mathcal{A}$-equivalent to $f_2$
at a singular point $p\in \Sigma$.
\end{corollary}

\subsection{Extrinsicity of invariants}
\label{sec:extrinsic}

In \cite{NUY}, it is shown that
the limiting normal curvature $\kappa_\nu$ 
of generic real analytic cuspidal edges
is an extrinsic invariant,
by proving the existence of non-trivial isometric deformations
of such cuspidal edges.
In this subsection, as an application of our Bour-type theorem
(Theorems \ref{thm:CE-3} and \ref{thm:singular-Bour}),
we prove that 
several invariants of generic helicoidal finite type edges 
are extrinsic invariants
(Corollaries \ref{cor:nu-c} and \ref{cor:kappa-c}).

Let 
$\CC=\CC_{[U,h,m,\epsilon_0,\epsilon_1,\epsilon_2]}
:J\times \R\to \R^3$
be the generic helicoidal $n$-type edge
given in Theorem \ref{thm:CE-3},
where $n=k+1$.
The first fundamental form $\I$ of
$\CC$
is given by \eqref{eq:metric-Bour-k}.
Since $\I$ does not depend on $h,m$,
we can easily determine that
a given invariant of $\CC$
is intrinsic or not.
In fact, as seen in Lemmas \ref{lem:LNC} and \ref{lem:cuspidal-torsion},
the limiting normal curvature $\kappa_\nu$
and the cusp-directional torsion $\kappa_t$
of $\CC$ at a singular point $(0,t)$
depend on $h,m$.
Hence, we have the following.

\begin{corollary}\label{cor:nu-c}
The limiting normal curvature $\kappa_\nu$
and
the cusp-directional torsion $\kappa_t$
of generic helicoidal finite type edges
are extrinsic.
\end{corollary}

In \cite{MSST}, invariants 
$\omega_{n,n+i}$ $(i=1,\dots,n-1)$
and $\beta_{n,2n}$
of $n$-type edges were introduced.
The invariant
$\omega_{n,n+i}$ is called the 
{\it $(n,n+i)$-cuspidal curvature\/}\footnote{%
In the case of $n=2$,
the $(2,3)$-cuspidal curvature $\omega_{2,3}$ was originally introduced 
in \cite{MSUY}
called the {\it cuspidal curvature\/} $\kappa_c$.
},
and 
$\beta_{n,2n}$ is called the {\it $(n,2n)$-bias\/}\footnote{%
In the case of $n=2$,
the $(2,4)$-bias $\beta_{2,4}$ was originally introduced 
in \cite{Honda-Saji}
called the {\it bias\/} $r_b$.
}.
We here review their definitions.
For a frontal $f : \Sigma \to \R^3$
having $n$-type edge at $p\in \Sigma$,
let $\xi$, $\eta$ be vector fields satisfying 
\eqref{itm:cri1} and \eqref{itm:cri2}.
For $i=1,\dots, n$, we set 
$$
   \omega_{n,n+i}(p) 
  = \frac{\|\xi{f}\|^{(n+i)/n}\det(\xi{f},\eta^{n}{f},\eta^{n+i}f)}
    {\|\xi{f}\times\eta^{n}{f}\|^{(2n+i)/n}}(p),
$$
Then, 
$\omega_{n,n+1}(p)$ does not depend on the choice of 
the pair $\xi$, $\eta$ of vector fields satisfying 
\eqref{itm:cri1} and \eqref{itm:cri2}.
For $i>1$,  if $\omega_{n,n+j}$ $(j=1,\dots,i-1)$ 
are identically zero on $S(f)$,
$\omega_{n,n+i}(p)$ does not depend on the choice of 
the pair $\xi$, $\eta$ of vector field satisfying 
\eqref{itm:cri1} and \eqref{itm:cri2}.
Then, $\omega_{n,n+i}(p)$ $(i=1,\dots,n-1)$ is called the 
{\it $(n,n+i)$-cuspidal curvature} of $f$ at $p$,
and 
$\omega_{n,2n}(p)$ is denoted by $\beta_{n,2n}(p)$
which is called the {\it $(n,2n)$-bias\/} of $f$ at $p$.
It holds that an $n$-type edge $f$ is a front at $p$
if and only if $\omega_{n,n+1}(p)\ne0$ (\cite{MSST}).
Here, we calculate 
the $(n,n+i)$-cuspidal curvature $\omega_{n,n+1}$
and 
the $(n,2n)$-bias $\beta_{n,2n}$
of $\CC$.

\begin{proposition}
\label{prop:cuspidal-curvature}
The $(n,n+i)$-cuspidal curvature $\omega_{n,n+i}$
$(i=1,\dots,n-1)$
and the $(n,2n)$-bias $\beta_{n,2n}$
of the generic helicoidal $n$-type edge
$\CC=\CC_{[U,h,m,\epsilon_0, \epsilon_1, \epsilon_2]}$
along the singular curve $\sigma(t)=(0,t)$
are respectively given by 
\begin{align}
\label{eq:omega-n}
  \omega_{n,n+i} 
  &= \epsilon_1 \epsilon_2 
  \frac{m^2 U(0) U^{(n+i)}(0)}{((n-1)!)^{(n+i)/n} \rho_{h,m} (0) },\\
\label{eq:beta-n}
\beta_{n,2n}
&= 
\frac{\epsilon_1 \epsilon_2}{\rho_{h,m} (0)}
\left(\frac{m^2 U(0) U^{(2n)}(0)}{((n-1)!)^2}-
\begin{pmatrix} 2 n-1\\ n\end{pmatrix}
\frac{h^2}{m^2 U(0)^2}\right),
\end{align}
where $\begin{pmatrix} 2 n-1\\ n\end{pmatrix}$ 
is the binomial coefficient.
\end{proposition}

\begin{proof}
We use the same notation as in 
the proof of Theorem \ref{thm:CE-3} and 
Lemma \ref{lem:cuspidal-torsion}.
We set $n=k+1$.
Since $\eta = \partial_s$ is a null vector field 
and $\xi = \partial_t$ is tangent to the singular set, 
\begin{equation}\label{eq:omega-i}
  \omega_{n,n+i}
  = \frac{1}{(k!)^{(2k+2+i)/(k+1)}\,U(0)}
  \det(\CC_{t} , \CC_{s^{k+1}} , \CC_{s^{k+1+i}})(0,t)
\end{equation}
holds,
where we used \eqref{eq:inner-product}.

With respect to $\omega_{n,n+1}$, 
as $\CC_s=s^k\overline{\CC}$, 
we have
$\CC_{s^{k+1}}(0,t)=k! \overline{\CC}(0,t)$
and 
$\CC_{s^{k+2}}(0,t)=U^{(k+2)}(0)\vect{w}(0,t)$
at $s=0$,
where we set 
$\overline{\CC}$ as \eqref{eq:bar-Psi},
and 
$$
 \vect{w}(s,t)=\frac{m^2 U}{\rho x}\Bigl(
        \rho \cos \theta - \epsilon_2h V \sin \theta,~
        \rho \sin \theta + \epsilon_2h V \cos \theta,~
        -\epsilon_2 m V x\Bigr).
$$
Then,
$
  \det(\CC_{t} , \CC_{s^{k+1}},\CC_{s^{k+2}})
  =\epsilon_1\epsilon_2k!m^2U^2U^{(k+2)}/\rho
$
holds at $s=0$.
Together with \eqref{eq:omega-i},
we obtain \eqref{eq:omega-n} with $i=1$.
With respect to $\omega_{n,n+2}$, 
by the condition $\omega_{n,n+1}=0$,
we have $U^{(k+2)}(0)=0$.
Then $\CC_{s^{k+3}}(0,t)=U^{(k+3)}(0)\vect{w}(0,t)$ holds,
and hence,
we obtain \eqref{eq:omega-n} with $i=2$.
Continuing this argument, 
for $i=2,\dots, n-1$,
we have that
$\omega_{n,n+1}=\dots=\omega_{n,n+i-1}=0$
if and only if 
$U^{(k+2)}=\dots=U^{(k+i)}=0$
at $s=0$.
Then $\CC_{s^{k+1+i}}=U^{(k+1+i)}\vect{w}$ holds
at $s=0$,
and hence,
we obtain \eqref{eq:omega-n}.

With respect to $\beta_{n,2n}$,
the condition
$\omega_{n,n+1}=\dots=\omega_{n,n+i}=0$
yields 
$U^{(k+2)}(0)=\dots=U^{(2k+1)}(0)=0$.
Let us set
$$
 \vect{a}(s,t)=\frac{1}{mU x}\left(
        a_1 \cos \theta +\frac{ha_2}{x} \sin \theta,~
        a_1 \sin \theta -\frac{ha_2}{x} \cos \theta,~
        a_2\right),
$$
where 
$
a_1=\frac{k!}{mUx^2}\left( 2m^4U^2V^2x^2-h^2\rho^2\right)$,
$a_2=\frac{\epsilon_2 m^2 k!UV}{x\rho}\left(2\rho^2+m^2h^2V^2\right).
$
Then,
$$
\CC_{s^{2k+2}}(0,t)
\equiv 
U^{(2k+2)}(0)\vect{w}(0,t)
+\frac{(2k+1)!}{(k+1)!}\vect{a}(0,t)
\mod \overline{\Psi}(0,t)
$$
holds.
Calculating
$\beta_{n,2n}
= \frac{1}{(k!)^3\,U(0)}
\det(\CC_{t} , \CC_{s^{k+1}} , \CC_{s^{2k+2}})(0,t)$,
we have \eqref{eq:beta-n}.
\end{proof}

\begin{corollary}\label{cor:kappa-c}
The $(n,n+i)$-cuspidal curvature $\omega_{n,n+i}$
and
the $(n,2n)$-bias $\beta_{n,2n}$
of generic helicoidal $n$-type edges
are extrinsic invariants.
\end{corollary}

\appendix
\section{Criterion for $7/2$-cusp}
\label{sec:7/2} 

In this appendix, we prove the following criterion for $7/2$-cusp
(Theorem \ref{thm:7/2-app}).
This criterion is based on the work by Takagi \cite{Takagi}.
In the following, the dot means $d/d t$ and we set 
$\ga^{(j)}:=d^{j}\ga/dt^{j}.$

\begin{theorem}\label{thm:7/2-app}
Let $\gamma(t)$ be a plane curve.
Then, $\gamma(t)$ has $7/2$-cusp at $t=c$ if and only if 
there exist real numbers $k, l \in \R$
such that
\begin{align}
\label{eq:27-1}
&\dot{\gamma}(c)=\vect{0},\quad
\ddot{\gamma}(c)\ne\vect{0},\quad
\dddot{\ga}(c) = k\ddot{\ga}(c),\quad
\ga^{(5)}(c)-\frac{10}{3}k \ga^{(4)}(c)=l\ddot{\ga}(c),\\
\label{eq:27-2}
&\det\left(\ddot{\ga}(c), \ga^{(7)}(c)-7k \ga^{(6)}(c)-\left( 7l-\frac{70}{3}k^3\right)\ga^{(4)}(c) \right) \ne 0.
\end{align}
\end{theorem}

First, we prove this criterion is invariant under 
coordinate changes in the domain.
In the following, the prime means $d/d u$ and we set 
$\ga^{[j]}:=d^{j}\ga/du^{j}.$

\begin{lemma}\label{lem:invariant-domain}
Let $\gamma(t)$ $(t\in I)$ be a plane curve satisfying 
\eqref{eq:27-1} and \eqref{eq:27-2}
for real numbers $k, l \in \R$.
Let $t = \varphi(u)$ be a diffeomorphism 
from an interval $J$ to $I$.
We set $\bar{c} := \varphi(c)$ and $\bar{\gamma}(u) := \gamma \circ t(u)$.
Then, there exist real numbers $\bar{k}, \bar{l} \in \R$
such that
\begin{align}
\label{eq:not35-u}
&\bar{\gamma}'(\bar{c})=\vect{0},\quad
\bar{\gamma}''(\bar{c})\ne\vect{0},\quad
\bar{\gamma}'''(\bar{c}) = \bar{k}\bar{\gamma}''(\bar{c}),\quad
\bar{\gamma}'^{[5]}(\bar{c})
-\frac{10}{3}\bar{k} \bar{\gamma}^{[4]}(\bar{c})=\bar{l}\bar{\gamma}''(\bar{c}),\\
\label{eq:27-u}
&\det\left(\bar{\gamma}''(\bar{c}), 
\bar{\gamma}^{[7]}(\bar{c})
-7\bar{k} \bar{\gamma}^{[6]}(\bar{c})
-\left( 7\bar{l}-\frac{70}{3}\bar{k}^3\right)\bar{\gamma}^{[4]}(\bar{c}) \right) \ne 0.
\end{align}
\end{lemma}

\begin{proof}
Without loss of generality, we may assume that $c=\bar{c}=0$.
Since $\bar{\gamma}'(u)=\varphi'(u)\dot{\ga}(\varphi(u))$,
we have $\bar{\gamma}'(0)=\vect{0}$.
As $\bar{\gamma}''(0)=\varphi'(0)^2\ddot{\ga}(0)$
and 
$\varphi'(0)$ does not vanish, 
$\bar{\gamma}''(0)\ne\vect{0}$ holds.
By a direct calculation, 
we have
$\bar{\ga}'''(0)
=\varphi'(0)^3\dddot{\ga}(0)+3\varphi'(0)\varphi''(0)\ddot{\ga}(0)$.
Since $\dddot{\ga}(0)=k\ddot{\ga}(0)$, we obtain
\begin{equation}\label{k-tilde}
\bar{\ga}'''(0)=\bar{k}\bar{\ga}''(0)\qquad
\left(
\bar{k}:=k\varphi'(0)+3\frac{\varphi''(0)}{\varphi'(0)}
\right).
\end{equation}
Hence we have the first, second and third equations
in \eqref{eq:not35-u}.

With respect to the fourth equation in \eqref{eq:not35-u},
by direct calculation, we have
\begin{align*}
\bar{\ga}^{[4]}(0)
&= (\varphi')^4\ga^{(4)} +\bigl(6k(\varphi')^2\varphi''+3(\varphi'')^2+4\varphi'\varphi'''\bigr)\ddot{\ga},\\
\bar{\ga}^{[5]}(0)
&=(\varphi')^5\ga^{(5)}+10(\varphi')^3\varphi''\ga^{(4)} 
\\&\hspace{10mm} 
+\bigl(15k\varphi'(\varphi'')^2 + 10k(\varphi')^2\varphi'''+10\varphi''\varphi'''+5\varphi'\varphi^{[4]}\bigr)\ddot{\ga},
\end{align*}
where the right hand sides are evaluated at $t=u=0$.
Hence, setting
\begin{multline}\label{l-tilde}
\bar{l}=
l(\varphi')^3-20k^2\varphi'\varphi''-k\left(55\frac{(\varphi'')^2}{\varphi'}+\frac{10}{3}\varphi'''\right)\\
-\frac{5}{(\varphi')^3}\bigl(6(\varphi'')^3+6\varphi'\varphi''\varphi'''-(\varphi')^2\varphi^{[4]}\bigr),
\end{multline}
where the right hand side is evaluated at $u=0$,
we have the fourth equation in \eqref{eq:not35-u}.

Next, we prove \eqref{eq:27-u}.
In the following,
$A\equiv B$
means that
$A- B$ is a scalar multiple of 
$\ddot{\ga}(0)$.
By direct calculation,
{\allowdisplaybreaks
we have
\begin{align*}
\bar{\ga}^{[6]}(0)
&\equiv(\varphi')^6\ga^{(6)}(0)+\bigl(50k(\varphi')^4\varphi''+
45(\varphi'\varphi'')^2+20(\varphi')^3\varphi'''\bigr)\ga^{(4)}(0),\\
\bar{\ga}^{[7]}(0)
&\equiv(\varphi')^7\ga^{(7)}(0)
+21(\varphi')^5\varphi''\ga^{(6)}(0)
+\biggl(350k(\varphi')^3(\varphi')^2
+\frac{350}{3}k(\varphi')^4\varphi'''
\\&\hspace{10mm} 
+105\varphi'(\varphi'')^3
+210(\varphi')^2\varphi''\varphi'''
+35(\varphi')^3\varphi^{[4]}\biggr)\ga^{(4)}(0),
\end{align*} 
where the right hand sides are evaluated at $u=0$.
Then, we obtain 
\begin{align*}
&\det\left(\bar{\ga}''(0) , 
\bar{\ga}^{[7]}(0)-7\bar{k}\bar{\ga}^{[6]}(0)
-\left(7\bar{l}-\frac{70}{3}\bar{k}^3\right)\bar{\ga}^{[4]}(0)\right) \\
&=(\varphi'(0))^9\det\Bigl(\ddot{\ga}(0) ,  \ga^{(7)}(0)-7k\ga^{(6)}(0)
-\Bigl(7l-\frac{70}{3}k^3\Bigr)\ga^{(4)}(c)\Bigr) \ne 0.
\end{align*}}
Therefore, we have \eqref{eq:27-u}
and \eqref{eq:not35-u}.
\end{proof}

Next, we prove the criterion in Theorem \ref{thm:7/2-app}
is invariant under 
coordinate changes in the target space $\R^2$.

\begin{lemma}\label{lem:invariant-target}
Let $\gamma(t)$ $(t\in I)$ be a plane curve satisfying 
\eqref{eq:27-1} and \eqref{eq:27-2}
for real numbers $k, l \in \R$.
Let $\Phi $ be a local diffeomorphism of $\R^2$.
We set $\Gamma(t) := \Phi\circ \gamma(t)$.
Then, 
\begin{align}
\label{eq:not35-target}
&\dot{\Gamma}(c)=\vect{0},\quad
\ddot{\Gamma}(c)\ne\vect{0},\quad
\dddot{\Gamma}(c) = k\ddot{\Gamma}(c),\quad
\Gamma^{(5)}(c)
-\frac{10}{3}k \Gamma^{(4)}(c)=l\ddot{\Gamma}(c),\\
\label{eq:27-target}
&\det\left(\ddot{\Gamma}(c), 
\Gamma^{(7)}(c)
-7k \Gamma^{(6)}(c)
-\left( 7l-\frac{70}{3}k^3\right)\Gamma^{(4)}(c) \right) \ne 0.
\end{align}
\end{lemma}

\begin{proof}
Without loss of generality, we may assume that $c=0$.
We denote by $\Phi(x,y)=(\Phi _{1}(x, y), \Phi _{2}(x, y) )$
the local diffeomorphism $\Phi $ of $\R^2$.
The Jacobian matrix of $\Phi(x, y)$ is written as
$$
J(x, y)
=
\begin{pmatrix} \partial \Phi_{1}/\partial x & 
\partial \Phi_{1}/\partial y\\
\partial \Phi_{2}/\partial x & \partial \Phi_{2}/\partial y
\end{pmatrix}.
$$
Setting $J(t):=J \circ \ga(t)$,
we have $\dot{\Gamma}(t)=J(t)\dot{\ga}(t)$,
and hence, $\dot{\Gamma}(0)=\vect{0}$ holds.
Since 
\begin{equation}\label{eq:Jdot}
\dot{J}(t)
=\dot{x}(t)J_{x}(t)+\dot{y}(t)J_{y}(t),
\end{equation} 
we obtain $\dot{J}(0)=0$,
$
\ddot{\Gamma}(0)=J(0)\ddot{\ga}(0)$,
and 
$
\dddot{\Gamma}(0)=J(0)\dddot{\ga}(0).
$
By \eqref{eq:27-1}, we have 
$
\dddot{\Gamma}(0)=k\,\ddot{\Gamma}(0).
$
Next, by a calculation similar to that in \eqref{eq:Jdot},
$\ddot{J}(0) = J_{x}(0)\ddot{x}(0)+J_{y}(0)\ddot{y}(0)$
holds.
By a direct calculation, we have
\begin{equation}\label{eq:G45}
\Gamma^{(4)}(0)=J\ga^{(4)}+3\ddot{J}\ddot{\ga},\qquad
\Gamma^{(5)}(0)=J\ga^{(5)}+6\ddot{J}\dddot{\ga}+4\dddot{J}\ddot{\ga},
\end{equation}
where the right hand sides are evaluated at $t=0$.
By $\dot{x}(0) = \dot{y}(0) = 0$, we obtain
$\ddot{J}(0) = J_{x}(0)\ddot{x}(0)+J_{y}(0)\ddot{y}(0)$.
Since $\dot{J}(0)=\dot{J_{x}}(0)=\dot{J_{y}}(0)=0$,
we have 
\begin{equation}\label{eq:J-ddd}
\dddot{J}(0) = k\ddot{J}(0).
\end{equation}
By substituting this into $\Gamma^{(5)}(0)$ in \eqref{eq:G45}, 
we obtain 
$
\Gamma^{(5)}(0)=J(0)\ga^{(5)}(0)+10k\ddot{J}(0)\ddot{\ga}(0).
$
By \eqref{eq:27-1}, 
$$
\Gamma^{(5)}(0)-\frac{10}{3}k \Gamma^{(4)}(0)
=l\,\ddot{\Gamma}(0)
$$
holds, and hence, we have \eqref{eq:not35-target}.

With respect to \eqref{eq:27-target},
since $\dot{\ga}(0)=\vect{0}$, $\dot{J}(0)=0$, $\dddot{J}(0)=k\ddot{J}(0)$, 
and $\dddot{\ga}(0)=k\ddot{\ga}(0)$, 
we have 
\begin{align}
\label{eq:Ga-6}
\Gamma^{(6)}(0) 
&= J\ga^{(6)}+10\ddot{J}\ga^{(4)} +10k^2\ddot{J}\ddot{\ga}+5J^{(4)}\ddot{\ga},\\
\label{eq:Ga-7}
\Gamma^{(7)}(0) 
&= J\ga^{(7)}+15\ddot{J}\ga^{(5)}+20k\ddot{J}\ga^{(4)} +15kJ^{(4)}\ddot{\ga}+6J^{(5)}\ddot{\ga},
\end{align}
where the right hand sides are evaluated at $t=0$.
In a similar way as in \eqref{eq:J-ddd},
we have
$\dddot{J_{x}}(0)=k\ddot{J_{x}}(0)$, 
$\dddot{J_{y}}(0)=k\ddot{J_{y}}(0)$.
Since 
\begin{align*}
J^{(4)}
&=J_{x}x^{(4)} +J_{y}y^{(4)}+3\ddot{J_{x}}\ddot{x}+3\ddot{J_{y}}\ddot{y},\\
J^{(5)}
&=J_{x}x^{(5)}+6\ddot{J_{x}}\dddot{x}
+4\dddot{J_{x}}\ddot{x}+J_{y}y^{(5)}+6\ddot{J_{y}}\dddot{y}+4\dddot{J_{y}}\ddot{y}
\end{align*}
hold at $t=0$, we have
$$
 J^{(5)}(0) -\frac{10}{3}kJ^{(4)}(0)=l\ddot{J}(0).
$$
Hence, by \eqref{eq:Ga-6} and \eqref{eq:Ga-7},  
\begin{align*}
  \Gamma^{(7)}-7k\Gamma^{(6)}-\left(7l-\frac{70}{3}k^3\right)\Gamma^{(4)}
  =J(0)\biggl( \ga^{(7)}-7k\ga^{(6)}-\left(7l-\frac{70}{3}k^3\right)\ga^{(4)} \biggr)
\end{align*}
holds at $t=0$.
By the condition \eqref{eq:27-2}, we obtain 
\begin{align*}
&\det\left(\ddot{\Gamma}(0) , \Gamma^{(7)}(0)-7k\Gamma^{(6)}(0)-\left(7l-\frac{70}{3}k^3\right)\Gamma^{(4)}(0)\right)  \\
&\hspace{5mm}
=\det(J(0))\det\left(\ddot{\ga}(0) , \ga^{(7)}(0)-7k\ga^{(6)}(0)-\left(7l-\frac{70}{3}k^3\right)\ga^{(4)}(0)\right) \ne 0.
\end{align*}
This completes the proof.
\end{proof}

We use the following fact to prove
Theorem \ref{thm:7/2-app}.

\begin{fact}[\cite{Ishikawa}] 
\label{lemma:2}
Let $\ga(t)$ be a $C^\infty$ curve in $\R^2$.
For a positive integer $n$,
if $\dot{\ga}(0)=0$, 
$\dddot{\ga}(0)=\cdots=\ga^{(2n)}(0)=0$ and
\begin{gather*}
 \det(\ddot{\ga}(0),\ga^{(2n+1)}(0)) \ne 0
\end{gather*}
hold,
then $\ga(t)$ has $(2,2n+1)$-cusp at $t=0$.
\end{fact}

\begin{proof}[Proof of Theorem \ref{thm:7/2-app}]
The standard $(2,7)$-cusp $\ga_{0}(t) = (t^2, t^7)$
satisfies \eqref{eq:27-1} and \eqref{eq:27-2} with $k=l=0$.
By Lemmas \ref{lem:invariant-domain} and \ref{lem:invariant-target},
we have that $\gamma(u):=\Phi\circ \ga_{0} (t(u))$ also satisfies 
\eqref{eq:27-1} and \eqref{eq:27-2}.
Conversely, 
let $\gamma(t)=(x(t), y(t))$ be a plane curve satisfying 
\eqref{eq:27-1} and \eqref{eq:27-2}.
It suffices to show that $\gamma(t)$
has $(2,7)$-cusp at $t=c$.

Without loss of generality, we may suppose that $c=0$
and $\ga(0)=\vect{0}$.
Since $\dot{\ga}(0)=\vect{0}$,
by the Hadamard's lemma (cf.\ \cite[Lemma 3.4]{BG}),
there exists a $C^{\infty}$ map $\vect{c}(t)=(a(t), b(t))$ 
such that
$$
  \ga(t)=t^2\vect{c}(t)\qquad
  \vect{c}(0)\ne (0,0)
$$
holds.
By rotating the plane, if necessary, 
we may assume that $a(0) \ne 0$.
We set $\Phi(x,y)$ as
$$
\Phi(x,y):= \left( x, y-\frac{m_1}{a(0)}x-\frac{m_2}{2a(0)^3}x^2-\frac{m_3}{24a(0)^5}x^3\right)
$$
respectively, where $m_1=b(0)$,
$m_2=\det( \vect{c}(0),\ddot{\vect{c}}(0))$,
and 
$$
m_3=a(0)^2\det( \vect{c}(0),\vect{c}^{(4)}(0))
-12(\ddot{a}(0)a(0)+\dot{a}(0)^2)\det(\vect{c}(0),\vect{c}^{(4)}(0)).
$$
If we set 
$\Gamma(t):=\Phi\circ\gamma(t)$,
then we may write
$$
\Gamma(t)=(a(t),\beta(t)),\qquad
\beta(t) = 
b(t)-\frac{m_1}{a(0)}a(t)-\frac{m_2}{2a(0)^3}a(t)^2-\frac{m_3}{24a(0)^5}a(t)^3.
$$
By a direct calculation, we have
$\beta^{(i)}(0)=0$
$(i=0,1,2,4,6)$.
Then \eqref{eq:27-1} and \eqref{eq:27-2}
yield that $\dddot{\beta}(0)=\beta^{(5)}(0)=0$ and $\beta^{(7)}(0)\ne0$.

In summary, it holds that
$\beta^{(i)}(0)=0$ $(0\le i\le 6)$ and $\beta^{(7)}(0)\ne0$.
Hence, applying the Hadamard's lemma again, 
there exists a $C^{\infty}$ function $B(t)$ 
such that
$$
  \ga(t)=(t^2a(t),t^7B(t))\qquad
  (a(0)\ne 0, B(0)\ne 0).
$$
By Fact \ref{lemma:2}, we have that 
$\gamma(t)$ has $(2,7)$-cusp at $t=0$.
\end{proof}

\section*{Statements and Declarations}

\subsubsection*{Acknowledgements}
The authors would like to 
thank Toshizumi Fukui, Yoshiki Matsushita and Masaaki Umehara
for helpful comments.
This work was supported by the Research Institute for Mathematical
Sciences, an International Joint Usage/Research Center located in Kyoto
University.

\subsubsection*{Funding}
This work was supported by 
the Japan Society for the Promotion of Science (JSPS) KAKENHI Grants 
JP19K14526 and JP20H01801.

\subsubsection*{Data availability}
Data sharing not applicable to this article as no datasets were generated or analysed during the current study.

\subsubsection*{Conflict of interest}
The authors declare that they have no conflict of interest.



\begin{thebibliography}{20}

\bibitem{Bour}
E.~Bour, 
{\it Memoire sur le deformation de surfaces}, 
Journal de l'Ecole Polytechnique, 
XXXIX Cahier, 1862, 1--148

\bibitem{BG}
  J.W.~Bruce and P.J.~Giblin,
  Curves and Singularities: 
  A Geometrical Introduction to Singularity Theory,
  Cambridge Univ.\ Press, 1992.

\bibitem{Bruce-Gaffney}
  J.W.~Bruce and T.J.~Gaffney,
  {\it Simple singularities of mappings \({\mathbb{C}},0\to {\mathbb{C}}^ 2,0\)},
  J. Lond. Math. Soc., II. Ser. {\bf 26} (1982), 465--474.


\bibitem{COP}
R.~Caddeo, I.I.~Onnis and P.~Piu, 
{\it Bour's theorem and helicoidal surfaces 
with constant mean curvature in the Bianchi-Cartan-Vranceanu spaces}, 
Ann. Mat. Pura Appl. {\bf 201} (2022), 913--932. 


\bibitem{doCarmo-Dajczer}
M.P.~do Carmo and M.~Dajczer, 
{\it Helicoidal surfaces with constant mean curvature},
Tohoku Math.\ J.\ (2) {\bf 34} (1982), 425--435.

\bibitem{Darboux}
G.\ Darboux,
Le{\c{c}}ons sur la th{\'e}orie g{\'e}n{\'e}rale des surfaces,
Vol. I, Paris, 1914.

\bibitem{DOP}
I.~Domingos, I.I.~Onnis and P.~Piu,
{\it The Bour's Theorem for invariant surfaces in three-manifolds},
preprint (arXiv:2306.03837).

\bibitem{Fuchs-Tabach}
D.~Fuchs and S.~Tabachnikov, 
{\it More on paper folding}, 
The American Mathematical Monthly {\bf 106} (1999), 27--35.

\bibitem{fukui2}
T.~Fukui,
{\it Local differential geometry of singular curves with finite multiplicities},
Saitama Math. J. {\bf 31} (2017), 79--88.


\bibitem{HHNSUY}
M. Hasegawa, A. Honda, K. Naokawa, K. Saji, M. Umehara and K. Yamada,
{\it Intrinsic properties of surfaces with singularities},
Internat. J. Math. {\bf 26} (2015), no. 4, 1540008, 34 pp.

\bibitem{Honda}
A.~Honda, 
{\it On associate families of spacelike Delaunay surfaces}, 
Real and complex singularities, 103--120, 
Contemp. Math., 675, Amer. Math. Soc.,
Providence, RI, 2016.


\bibitem{HNUY}
  A.~Honda, K.~Naokawa, M.~Umehara and K.~Yamada, 
  {\it Isometric deformations of wave fronts 
             at non-degenerate singular points}, 
  Hiroshima Math. J. {\bf 50} (2020), 269--312.

\bibitem{HNSUY1}
  A.~Honda, K.~Naokawa, K.~Saji, M.~Umehara and K.~Yamada,
  {\it Duality on generalized cuspidal edges 
  preserving singular set images and first fundamental forms}, 
  J. Singul. {\bf 22} (2020), 59--91.

\bibitem{HNSUY2}
  A.~Honda, K.~Naokawa, K.~Saji, M.~Umehara and K.~Yamada,
  {\it Cuspidal edges with the same first fundamental forms along a knot}, 
  J. Knot Theory Ramifications {\bf 29} (2020), 
  Article ID 2050047, 16p.

\bibitem{HNSUY3}
  A.~Honda, K.~Naokawa, K.~Saji, M.~Umehara and K.~Yamada,
  {\it Curved foldings with common creases and crease patterns}, 
  Adv. Appl. Math. {\bf 121} (2020), Article ID 102083, 10p.

\bibitem{HNSUY4}
  A.~Honda, K.~Naokawa, K.~Saji, M.~Umehara and K.~Yamada,
  {\it On the existence of four or more curved foldings 
  with common creases and crease patterns}, 
  Beitr. Algebra Geom. {\bf 63} (2022), No. 4, 723--761.

\bibitem{Honda-Saji}
  A.~Honda and K.~Saji, 
  {\it Geometric invariants of $5/2$-cuspidal edges}, 
  Kodai Math.\ J.\ {\bf 42} (2019), 496--525.

\bibitem{Ishikawa}
G.~Ishikawa,
{\it Determinacy of the envelope of the osculating hyperplanes to a curve},
Bull. Lond. Math. Soc. {\bf 25} (1993), 603--610.

\bibitem{Ji-Kim}
F.~Ji and Y.H.~Kim, 
{\it Mean curvatures and Gauss maps of a pair of 
isometric helicoidal and rotation surfaces in Minkowski 3-space}, 
J. Math. Anal. Appl. {\bf 368} (2010), 623--635.

\bibitem{KRSUY}
 M. Kokubu,  W. Rossman, K. Saji, M. Umehara and K. Yamada,
 {\it Singularities of flat fronts in hyperbolic $3$-space},
 Pacific J. Math. {\bf 221} (2005), 303--351.

\bibitem{Kos}
M.~Kossowski, 
{\it Realizing a singular first fundamental form 
as a nonimmersed surface in Euclidean 3-space}, 
J.\ Geom.\ {\bf 81} (2004), 101--113.

\bibitem{Kos2}
M.~Kossowski, 
{\it The Boy-Gauss-Bonnet theorems for $C^\infty$-singular surfaces 
with limiting tangent bundle}, 
Ann.\ Global Anal.\ Geom.\ {\bf 21} (2002), 19--29.

\bibitem{MS}
 L. F. Martins and K. Saji,
 {\it Geometric invariants of cuspidal edges},
 Canad. J. Math. {\bf 68} (2016), no. 2, 445--462.

\bibitem{MS2}
 L. F. Martins and K. Saji,
 {\it Geometry of cuspidal edges with boundary},
 Topol.\ Appl.\ {\bf 234} (2018) 209--219.

\bibitem{MSUY}
L. F. Martins, K. Saji, M. Umehara and K. Yamada,
{\it Behavior of Gaussian curvature and
mean curvature near non-degenerate singular
points on wave fronts},
Geometry and Topology of Manifold,
Springer Proc.\ in Math.\ \& Stat. {\bf 154},
2016, Springer, 247--282.

\bibitem{MSST19}
L.F.~Martins, K.~Saji, S.P.~dos Santos and K. Teramoto,
{\it Singular surfaces of revolution with prescribed unbounded mean curvature},
An. Acad. Bras. Ci{\^e}nc. {\bf 91} (2019), Article ID e20170865, 10 p.

\bibitem{MSST}
 L.F.~Martins, K.~Saji, S.P.~dos Santos and K. Teramoto,
 {\it Boundedness of geometric invariants near a singularity
which is a suspension of a singular curve},
arXiv:2206.11487, 
to appear in Revista de la Uni{\'o}n Matem{\'a}tica Argentina.

\bibitem{Matsushita}
Y.~Matsushita,
{\it Classifications of cusps appearing on plane curves},
preprint (arXiv:2402.12166).

\bibitem{NUY}
K.~Naokawa, M.~Umehara and K.~Yamada,
{\it Isometric deformations of cuspidal edges},
  Tohoku Math.\ J.\ (2) {\bf 68} (2016), 73--90.

\bibitem{porteous}
I.R.~Porteous,
{\it Geometric differentiation for the intelligence of curves and surfaces},
Second edition, Cambridge University Press, Cambridge, 2001. 

\bibitem{SaEarp}
R.~Sa Earp, 
{\it Parabolic and hyperbolic screw motion surfaces 
in {{\(\mathbb H^{2}\times \mathbb R\)}}},
J. Aust. Math. Soc. {\bf 85} (2008), 113--143.

\bibitem{SaEarp-Toubiana}
R.~Sa Earp and E.~Toubiana,
{\it Screw motion surfaces in 
{{\(\mathbb H^2 \times \mathbb R\)}} 
and {{\(\mathbb S^2\times \mathbb R\)}}},
Ill. J. Math. {\bf 49} (2005), 1323--1362.

\bibitem{SUY-kyushu} 
  K.~Saji, M.~Umehara and K.~Yamada,
  {\it Behavior of corank one singular points on wave fronts},
  Kyushu J.\ Math.\ {\bf 62} (2008), 259--280.

\bibitem{SUY-ann}
  K.~Saji, M.~Umehara and K.~Yamada,
  {\it The geometry of fronts},
  Ann.\ of Math.\ (2) {\bf 169} (2009), 491--529.

\bibitem{SUY4}
  K.~Saji, M.~Umehara and K.~Yamada,
  {\it Coherent tangent bundles and {G}auss-{B}onnet formulas for wave fronts},
  J.\ Geom.\ Anal.\ {\bf 22} (2012), 383--409.


\bibitem{Sasahara}
N.~Sasahara, 
{\it Spacelike helicoidal surfaces with constant mean curvature 
in Minkowski 3-space},
Tokyo J. Math. {\bf 23} (2000), 477--502.

\bibitem{su}
S.~Shiba and M.~Umehara, 
{\it The behavior of curvature functions at cusps and inflection points}, 
Differential Geom.\ Appl.\ {\bf 30} (2012), 285--299.

\bibitem{Takagi}
N.~Takagi,
{\it Criterion for 7/2-cusp singularity} (in Japanese),
Bachelor thesis, Yokohama National University, 2021.

\bibitem{SUY-book}
  M.~Umehara, K.~Saji and K.~Yamada,
  Differential geometry of curves and surfaces with singularities. 
  Translated from the Japanese by Wayne Rossman. 
  Series in Algebraic and Differential Geometry 1. 
  Singapore: World Scientific. xvi, 370 p. (2022).





\end{thebibliography}
\end{document}